\documentclass[sts,
	reqno,
]{imsart}

\usepackage{amsthm, amssymb, amsmath}
\usepackage{mathtools}
\usepackage{color}
\usepackage{tikz}
\usetikzlibrary{arrows}
\usepackage[utf8]{inputenc}
\usepackage[shortlabels]{enumitem}
\usepackage{url}
\usepackage{etex}
\usepackage{alltt}
\usepackage{memhfixc}
\usepackage{bbold}%
\usepackage{cases}
\usepackage[noadjust]{cite}
\definecolor{labelkey}{rgb}{0.0, 0.8, 0.3}
\usepackage[top=2.5cm,bottom=2.5cm,left=2.5cm,right=2.5cm,includeheadfoot,asymmetric]{geometry}
\usepackage{algorithm}
\usepackage{algpseudocode}
\usepackage{wrapfig}

\numberwithin{equation}{section}
\usepackage{hyperref}
\hypersetup{
  colorlinks = true,
  urlcolor = blue, 
  linkcolor = blue,
  citecolor = blue}

\definecolor{ggreen}{rgb}{0.0, 0.5, 0.3}
\definecolor{rred}{rgb}{0.65, 0.2, 0.2}
\definecolor{bblue}{rgb}{0.0, 0.0, 1}

\newtheorem{thm}{Theorem}
\newtheorem*{thm*}{Theorem}
\newtheorem*{ass*}{Assumption}
\newtheorem{rem}[thm]{Remark}
\newtheorem{defi}[thm]{Definition}
\newtheorem{lem}[thm]{Lemma}
\newtheorem{cor}[thm]{Corollary}
\newtheorem{pro}[thm]{Proposition}

\newcommand{\cC}{\mathcal{C}}
\newcommand{\cS}{\mathcal{S}} 
\renewcommand{\P}{P}	
\newcommand{\Pn}{{\P}_n}	

\renewcommand{\sep}{\zeta} 
\newcommand{\E}{\mathbb{E}}

\newcommand{\gradF}{\nabla F}
\newcommand{\gradG}{\nabla G}
\newcommand{\bs}{b^\star}	

\newcommand{\cP}{\mathcal{P}}

\newcommand{\R}{\mathbb R}

\newcommand{\eps}{\varepsilon}

\renewcommand{\phi}{\varphi}
\renewcommand{\ge}{\geqslant}
\renewcommand{\le}{\leqslant}
\newcommand{\ud}{\mathrm{d}}
\DeclareMathOperator{\Tr}{Tr}
\DeclareMathOperator{\dom}{dom}
\DeclareMathOperator{\id}{id}
\DeclareMathOperator{\interior}{int}

\DeclareMathOperator{\supp}{supp}

\DeclareMathOperator{\var}{var}

\DeclareMathOperator{\Div}{div}
\DeclarePairedDelimiter\norm{\lVert}{\rVert}

\newcommand{\mbb}[1]{\mathbb{#1}}
\newcommand{\mc}[1]{\mathcal{#1}}
\newcommand{\msf}[1]{\mathsf{#1}}
\newcommand{\on}[1]{\operatorname{#1}}

\definecolor{pink}{cmyk}{0, 1, 0, 0} 
\definecolor{cutepinky}{cmyk}{0, 0.9, 0, 0} 

\newcommand{\ndpr}[1]{{\color{pink}{[PR: #1]}}}
\newcommand{\as}[1]{{\color{ggreen}[Austin]\, #1}}
\newcommand{\scnote}[1]{{\color{orange}[Sinho: #1]\,}}

\renewcommand{\abstractname}{Abstract}

\begin{document}

\begin{frontmatter}

	\title{Gradient descent algorithms for Bures-Wasserstein barycenters}
	\runtitle{Gradient descent for Bures-Wasserstein barycenters}

	\author{ \fnms{Sinho} \snm{Chewi}\thanksref{}\ead[label=sc]{schewi@mit.edu},
		\fnms{Tyler} \snm{Maunu}\ead[label=tm]{maunut@mit.edu},
		\fnms{Philippe} \snm{Rigollet}\ead[label=rigollet]{rigollet@math.mit.edu},
			~and~
		\fnms{Austin J.} \snm{Stromme}\ead[label=ajs]{astromme@mit.edu}
	}


	\address{{Sinho Chewi}\\
		{Department of Mathematics} \\
		{Massachusetts Institute of Technology}\\
		{77 Massachusetts Avenue,}\\
		{Cambridge, MA 02139-4307, USA}\\
		 \printead{sc}
	}
	
		\address{{Tyler Maunu}\\
		{Department of Mathematics} \\
		{Massachusetts Institute of Technology}\\
		{77 Massachusetts Avenue,}\\
		{Cambridge, MA 02139-4307, USA}\\
		 \printead{tm}
	}

	\address{{Philippe Rigollet}\\
		{Department of Mathematics} \\
		{Massachusetts Institute of Technology}\\
		{77 Massachusetts Avenue,}\\
		{Cambridge, MA 02139-4307, USA}\\
		 \printead{rigollet}
	}
	\address{{Austin Stromme}\\
		{Department of EECS} \\
		{Massachusetts Institute of Technology}\\
		{77 Massachusetts Avenue,}\\
		{Cambridge, MA 02139-4307, USA}\\
		\printead{ajs}
	}

\runauthor{Chewi, Maunu, Rigollet, \& Stromme}

\begin{abstract}
We study first order methods to compute the barycenter of a probability distribution $P$ over the space of probability measures with finite second moment. We develop a framework to derive global rates of convergence for both gradient descent and stochastic gradient descent despite the fact that the barycenter functional is not geodesically convex. Our analysis overcomes this technical hurdle by employing a Polyak-\L{}ojasiewicz (PL) inequality and relies on tools from optimal transport and metric geometry. In turn, we establish a PL inequality when $P$ is supported on the Bures-Wasserstein manifold of Gaussian probability measures. It leads to the first global rates of convergence for first order methods in this context. 



\end{abstract}



\end{frontmatter}
\section{Introduction}

We consider the following statistical problem. We observe $n$ independent copies $\mu_1,\dotsc,\mu_n$ of a random probability measure $\mu$ over $\R^D$. Assume furthermore that $\mu \sim \P$, where $P$ is an unknown distribution over probability measures. We wish to output a single probability measure on $\R^D$, $\bar \mu_n$, which represents the \emph{average} measure under $\P$ in a suitable sense. For example, the measures $\mu_1,\dotsc,\mu_n$ may arise as representations of images, in which case the average of the measures with respect to the natural linear structure on the space of signed measures is unsuitable for many applications~\cite{cuturi2014barycenters}.
Instead, we study the \emph{Wasserstein barycenter}~\cite{agueh2011barycenter} which has been proposed in the literature as a more desirable notion of average because it incorporates the geometry of the underlying space. Wasserstein barycenters have been applied in a broad variety of areas, including  graphics, neuroscience, statistics, economics, and algorithmic fairness~\cite{carlier2010matching, rabin2011wasserstein,  rabin_papadakis_2015,
solomon2015convolutional, gramfort2015fast, bonneel2016wasserstein, srivastava2018scalable, legouic2020fairness}.

To formally set up the situation, let $\mc P_2(\R^D)$ be the set of all (Borel) probability measures on $\R^D$ with finite second moment, and let $\mc P_{2,\rm ac}(\R^D)$ be the subset of those measures in $\mc P_2(\R^D)$ that are absolutely continuous with respect to the Lebesgue measure on $\R^D$ and thus admit a density. When endowed with the \textit{2-Wasserstein metric}, $W_2$, this set forms a geodesic metric space $(\mc P_{2,\rm ac}(\R^D),W_2)$.
We denote by $P_n$ the empirical distribution of the sample $\mu_1, \dotsc, \mu_n$.

A \emph{barycenter} of $\P$, denoted $\bs$, is defined to be a minimizer of the functional
\begin{align*}
    F(b) := \frac{1}{2} \P W_2^2(b, \cdot) = \frac{1}{2} \int W_2^2(b,\cdot) \, \ud \P.
\end{align*}
A natural estimator of $\bs$ is the \emph{empirical barycenter} $\hat b_n$, defined as a minimizer of
\begin{align*}
    F_n(b) := \frac{1}{2} \Pn W_2^2(b, \cdot) = \frac{1}{2n} \sum_{i=1}^n W_2^2(b, \mu_i).
\end{align*}

Statistical consistency of the empirical barycenter in a general context was first established in~\cite{legouic2015consistency} and further work has focused on providing effective rates of convergence for the quantity $W_2^2(\hat b_n, \bs)$.
A first step towards this goal was made in~\cite{ahidarcoutrix2018convergence} by deriving nonparametric rates of the form $W_2^2(\hat b_n, \bs) \lesssim n^{-1/D}$ when $D \ge 3$. Moreover, in the same paper~\cite{ahidarcoutrix2018convergence}, the authors establish parametric rates of the form $W_2^2(\hat b_n, \bs) \lesssim n^{-1}$  when $P$ is supported on a space of finite doubling dimension. 

An important example with this property arises when $P$ is supported on centered non-degenerate Gaussian measures, first studied by
Knott and Smith in 1994~\cite{knott1994generalization}.
In this case, Gaussians can be identified with their covariance matrices, and the Wasserstein metric induces a metric on the space of positive definite matrices. This metric, known as the \emph{Bures} or \emph{Bures-Wasserstein} metric is the distance function for a Riemannian metric on the manifold of positive definite matrices, known as the \emph{Bures manifold}~\cite{modin2017matrix,bhatia2019bures}.
The name of the Bures manifold originates from quantum physics and quantum information theory, where it is used to model the space of density matrices~\cite{bures1969extension}.
In the Bures case of the barycenter problem, more precise statistical results, including central limit theorems, are known~\cite{agueh2017centrale,kroshnin2019barycenters}. 

It is worth noting that parametric rates are also achievable in the infinite-dimensional case under additional conditions. First, it is not surprising that such rates are achievable over $(\mc P_{2}(\R),W_2)$ since this space can be isometrically
embedded in a Hilbert space~\cite{panaretos2016point, bigot2018barycenter}. Moreover, it was shown that, under additional geometric conditions, such rates are achievable for much more general infinite-dimensional spaces~\cite{legouic2019fast}, including $(\mc P_{2,\rm ac}(\R^D),W_2)$ for any $D\ge 2$.

While these  results are satisfying from a statistical perspective, they do not provide guidelines for the \emph{computation} of the empirical barycenter $\hat b_n$. In practice,
Wasserstein barycenters are estimated using iterative, first order algorithms~\cite{cuturi2014barycenters,esteban2016barycenters,backhoffveraguas2018barycenters,claici2018stochastic,zemel2019procrustes} but often lack theoretical guarantees. Recently, this line of work has provided rates of convergence for first order algorithms employed to compute the Wasserstein barycenter of distributions with a \emph{common discrete support}~\cite{guminov2019accelerated,pmlr-v97-kroshnin19a,Dvi20,lin2020fixedsupport}. In this framework, the computation of Wasserstein barycenters is a convex optimization problem with additional structure. However, first order methods can also be envisioned beyond this traditional framework by adopting a non-Euclidean perspective on optimization. This approach
is supported by the influential work of Otto~\cite{otto2001geometry} who established that Wasserstein space bears resemblance to a Riemannian manifold. In particular, one can define the Wasserstein gradient of the functional $F$, so it does indeed make sense to consider an intrinsic \emph{gradient descent}-based approach towards estimating $b^\star$. However, the convergence guarantees for such first order methods are largely unexplored.

When the distribution $P$ is supported on the Bures-Wasserstein manifold of Gaussian probability measures, gradient descent takes the form of a concrete and tractable update equation on the mean and covariance matrix of the candidate barycenter. In the population setting (where the distribution $P$ is known), such an algorithm was proposed in \'{A}lvarez-Esteban et al.~\cite{esteban2016barycenters}, where it is described as a fixed-point algorithm. \'{A}lvarez-Esteban et al.\ prove that the fixed-point algorithm converges to the true barycenter as the number of iterations goes to infinity. The consistency results were further generalized in~\cite{backhoffveraguas2018barycenters,zemel2019procrustes} and extended to the non-population and stochastic gradient case. However, the literature currently does not provide any rates of convergence for these first order methods.
In fact, \'{A}lvarez-Esteban et al.\ empirically observed a linear rate of convergence for the gradient descent algorithm in the Gaussian setting and left open the theoretical study of this phenomenon for future study. 
One contribution of this paper is to establish this rate of convergence (Theorem~\ref{thm:GD_short}), and we also provide multiple extensions including the first rate of convergence for stochastic gradient descent in this context.

On our way to proving rates of convergence in the Bures-Wasserstein case, we also establish results that apply to the more general setting where $P$ may not be supported on Gaussian probability measures. In particular, we establish an integrated  Polyak-\L{}ojasiewicz inequality (Lemma~\ref{lem:main}) and a new variance inequality (Theorem~\ref{thm:variance_ineq}) that are of independent interest.

\medskip {\sc Notation.} 
We denote the set of positive definite matrices by $\mbb S_{++}^D$, and the set of positive semidefinite matrices by $\mbb S_{+}^D$. We denote by $\lambda_1(\Sigma), \ldots, \lambda_D(\Sigma) \ge 0$ the eigenvalues of a matrix $\Sigma \in \mbb S_{+}^D$.
The Gaussian measure on $\R^D$ with mean $m\in \R^D$ and covariance matrix $\Sigma\in \mbb S_+^D$ is denoted $\gamma_{m,\Sigma}$. We reserve the notation $\log$ for the inverse of the Riemannian exponential map (which we review in~\ref{SEC:OT}) and use instead $\ln(\cdot)$ to denote the natural logarithm. The (convex analysis) indicator function $\iota_{\mathcal{C}}$ of a set $\mathcal{C}$ is defined by $\iota_{\mathcal{C}}(x)=0$ if $x \in \mathcal{C}$ and $\iota_{\mathcal{C}}(x)=+\infty$ otherwise. We denote by $\id$ the identity map of $\R^D$.

\section{Main results}

%
%


In this paper, we develop a general machinery to study first order methods for optimizing the barycenter functional on Wasserstein space. Establishing fast convergence of first order methods is usually intimately related
to convexity. Since our setting is on the curved Wasserstein space, we 
talk about \emph{geodesic convexity} rather than the usual notion convexity employed in flat, Euclidean spaces.
Geodesic convexity has been used to study statistical efficiency in manifold constrained estimation~\cite{AudMazRuh05, Wie12} and, more recently, in optimization~\cite{Bon13, Bac14,zhangSra16a}.

Barring a direct approach to establishing quantitative convergence guarantees, the barycenter functional is actually not geodesically convex on Wasserstein space. 
In fact, the barycenter functional may even be \emph{concave} along geodesics;  see Figure~\ref{fig:buresnoncvx}. As such, it does not lend itself to the general techniques of geodesically convex optimization. This non-convexity is a manifestation of the non-negative curvature 
of $(\mc P_2(\R^D),W_2)$ (cf. subsection~\ref{SEC:OT}) \cite{sturmnpc}.

Fortunately,
the optimization literature describes conditions for global convergence of first order algorithms even for non-convex objectives. In this work, we employ a Polyak-\L{}ojasiewicz (PL) inequality of the form \eqref{eq:pl}, which is known to yield linear convergence
for a variety of gradient methods on flat spaces even in absence of convexity~\cite{karimi2016linear}.\vskip 1em
\begin{wrapfigure}[9]{r}{0.4\textwidth}
   \vspace{-2.2em}
    \centering
    \includegraphics[width = .3\textwidth]{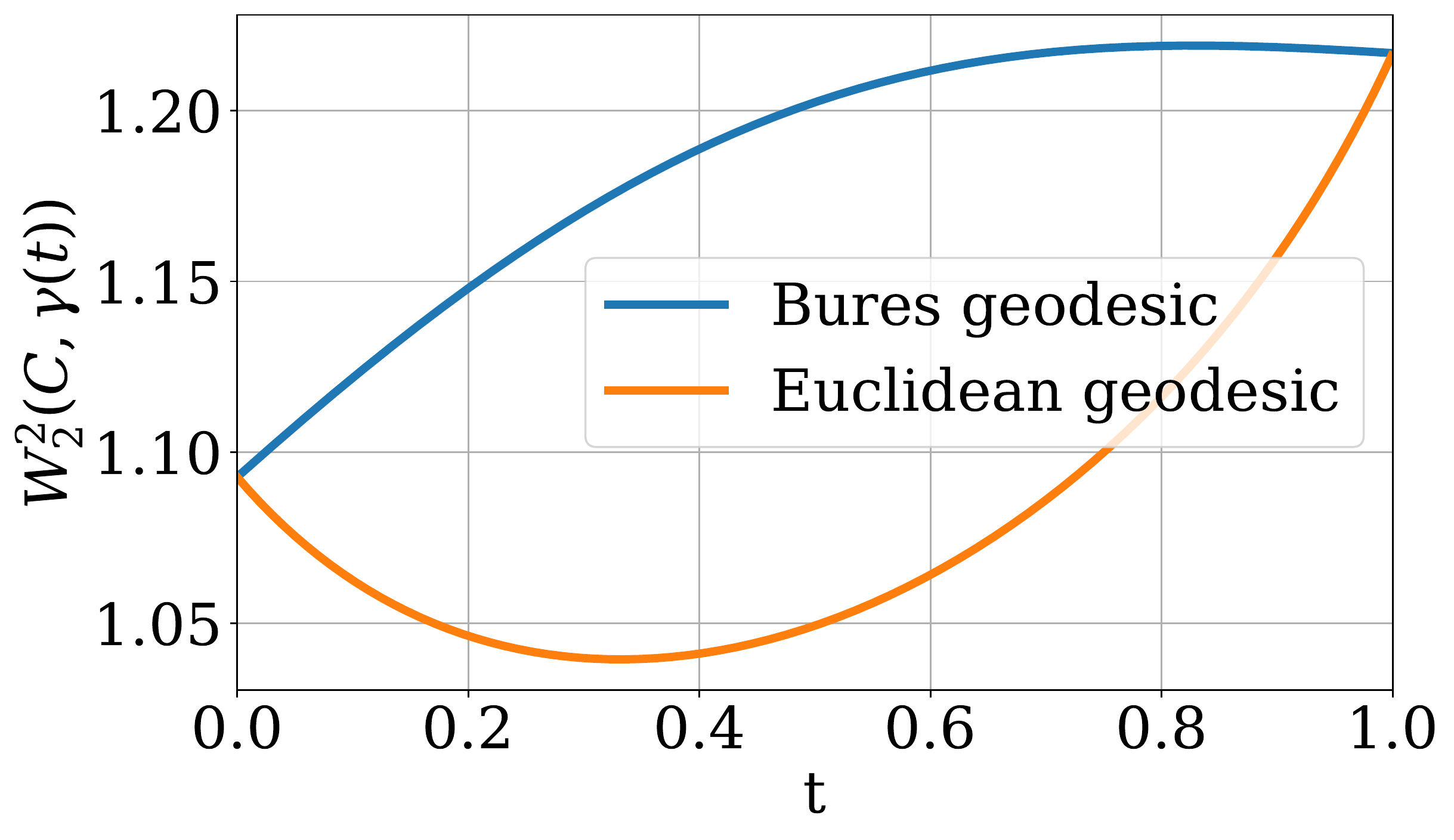}
    \caption{Example of the non-geodesic convexity of $W_2^2$.
    Details
    are given in Appendix~\ref{appendix:fig_details}.}
    \label{fig:buresnoncvx}
\end{wrapfigure}

In this paper, we study the barycenter functional
\begin{equation}
\label{EQ:bary_functional}
    G(b) := \frac{1}{2} Q W_2^2(b, \cdot) = \frac{1}{2} \int W_2^2(b,\cdot) \, \ud Q\,,
\end{equation}
for some generic distribution $Q$ with barycenter $\bar b$. This notation allows us to treat simultaneously the cases where $Q=\P$ and $Q=\P_n$, which are the situations of interest for statisticians. The case when $Q$ is an arbitrary discrete distribution supported on Gaussians has also been studied in the geodesic optimization literature~\cite{agueh2011barycenter, esteban2016barycenters, webersra2017rfw, bhatia2019bures, weber2019nonconvex, zemel2019procrustes}. Our main theorems, for gradient descent and stochastic gradient descent respectively, are stated below.


\begin{thm}\label{thm:GD_short}
Fix $\sep \in (0,1]$ and let $Q$ be a distribution supported on mean-zero Gaussian measures whose covariance matrices $\Sigma$ satisfy $\|\Sigma\|_{\rm op} \le 1$ and $\det \Sigma \ge \sep$. Then, $Q$ has a unique barycenter~$\bar b$, and Gradient Descent (Algorithm~\ref{ALG:GD}) initialized at $b_0 \in \supp(Q)$ 
yields a sequence ${(b_T)}_{T\ge 1}$ such that 
        \begin{align*}
            W_2^2(b_T, \bar b)
            \le \frac{2}{\sep} {\Bigl(1 - \frac{\sep^2}{4}\Bigr)}^T [G(b_0) - G(\bar b)]\,.
        \end{align*}
\end{thm}

The above theorem establishes a linear rate of convergence for gradient descent and answers a question left open in~\cite{esteban2016barycenters}. Moreover, when $Q=\P_n$, combined with the existing results of~\cite{kroshnin2019barycenters, ahidarcoutrix2018convergence}, it yields a procedure to estimate Wasserstein barycenters at the parametric rate after a number of iterations that is logarithmic in the sample size~$n$. 
%

Still in the Gaussian case, we also show that a stochastic gradient descent (SGD) algorithm converges to the true barycenter at a parametric rate. 
\begin{thm}\label{thm:SGD_short}Fix $\sep \in (0,1]$ and let $Q$ be a distribution supported on mean-zero Gaussian measures whose covariance matrices $\Sigma$ satisfy $\|\Sigma\|_{\rm op} \le 1$ and $\det \Sigma \ge \sep$. Then, $Q$ has a unique barycenter $\bar b$, and Stochastic Gradient Descent (Algorithm~\ref{ALG:SGD})
run on a sample of size $n+1$ from $Q$ returns a Gaussian measure $b_{n}$ such that
        \begin{align*}
            \E W_2^2(b_n, \bar b)
            \le \frac{96 \var(Q)}{n \zeta^5}\,,\quad \text{where} \quad 
\var(Q)=\int W_2^2(\cdot, \bar b) \, \ud Q\,.
        \end{align*}
\end{thm}
When applied to $Q=P$, Theorem~\ref{thm:SGD_short} shows that SGD yields an estimator $b_n$ different from the empirical barycenter $\hat b_n$ that also converges at the parametric rate to $b^\star$.  When applied to $Q=P_n$, this leads an alternative to gradient descent to estimate the empirical barycenter $\hat b_n$ that exhibits a slower convergence but that has much cheaper iterations and lends itself better to parallelization.
 
As far as we are aware, these results provide the first non-asymptotic rates of convergence for first order methods on the Bures-Wasserstein manifold.

\begin{rem}
A natural sufficient condition of $\det \Sigma \ge \sep$ to be  satisfied is when all the eigenvalues of the covariance matrix $\Sigma$ are lower bounded by a constant $\lambda_{\min} > 0$. In this case, the parameter $\sep \ge \lambda_{\min}^D$ can be exponentially small in the dimension. Note however that, in this case, the Gaussian measure is quite degenerate in the sense that the density of $\gamma_{0, \Sigma}$ is exponentially large at 0.\end{rem}

In Figure~\ref{fig:sgd}, we present the results an experiment confirming these two results; see Appendix~\ref{sec:experiments} for more details and further numerical results.

\begin{figure}[t]
    \centering
    \includegraphics[width = .49\textwidth]{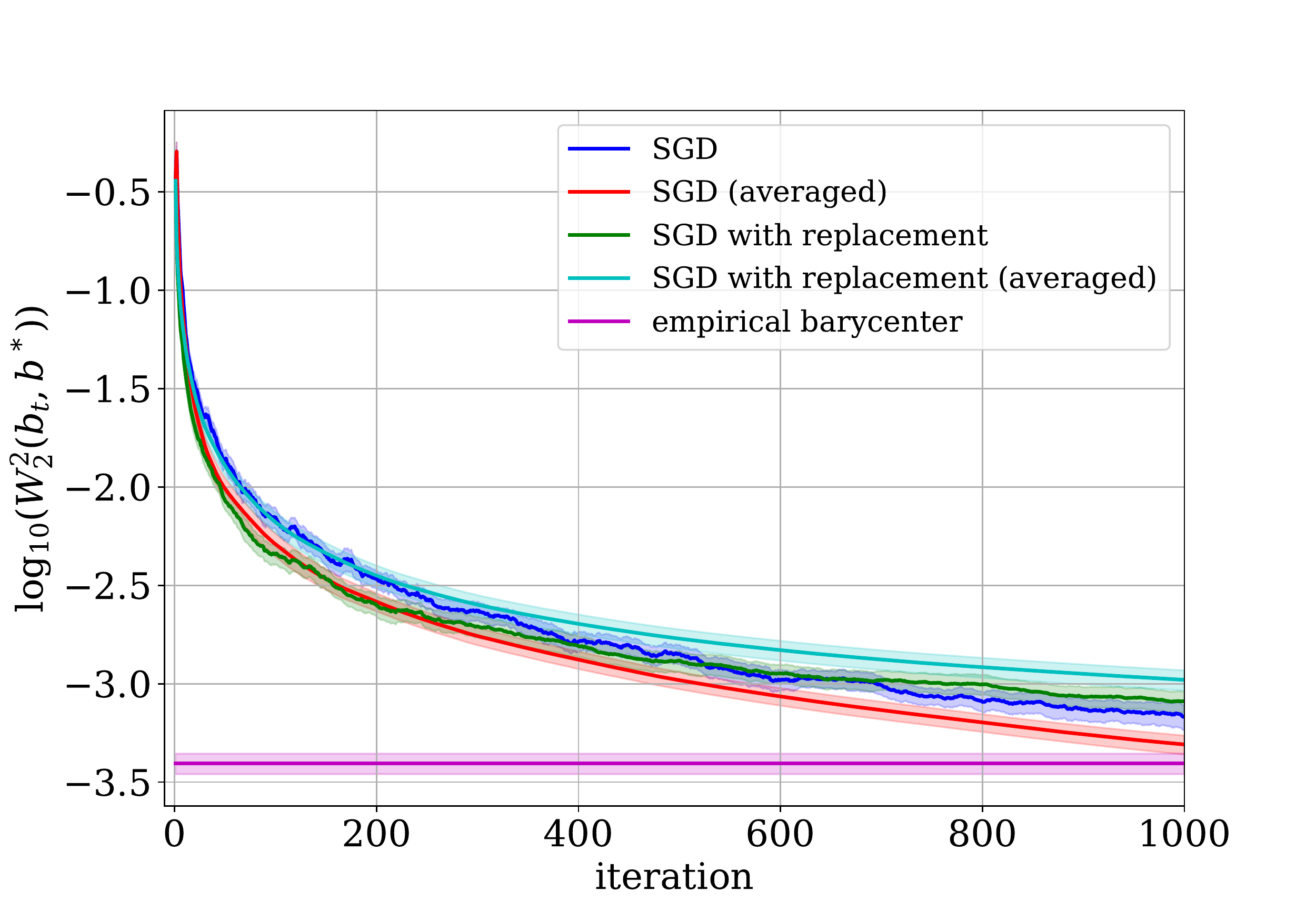}
    \includegraphics[width = .49\textwidth]{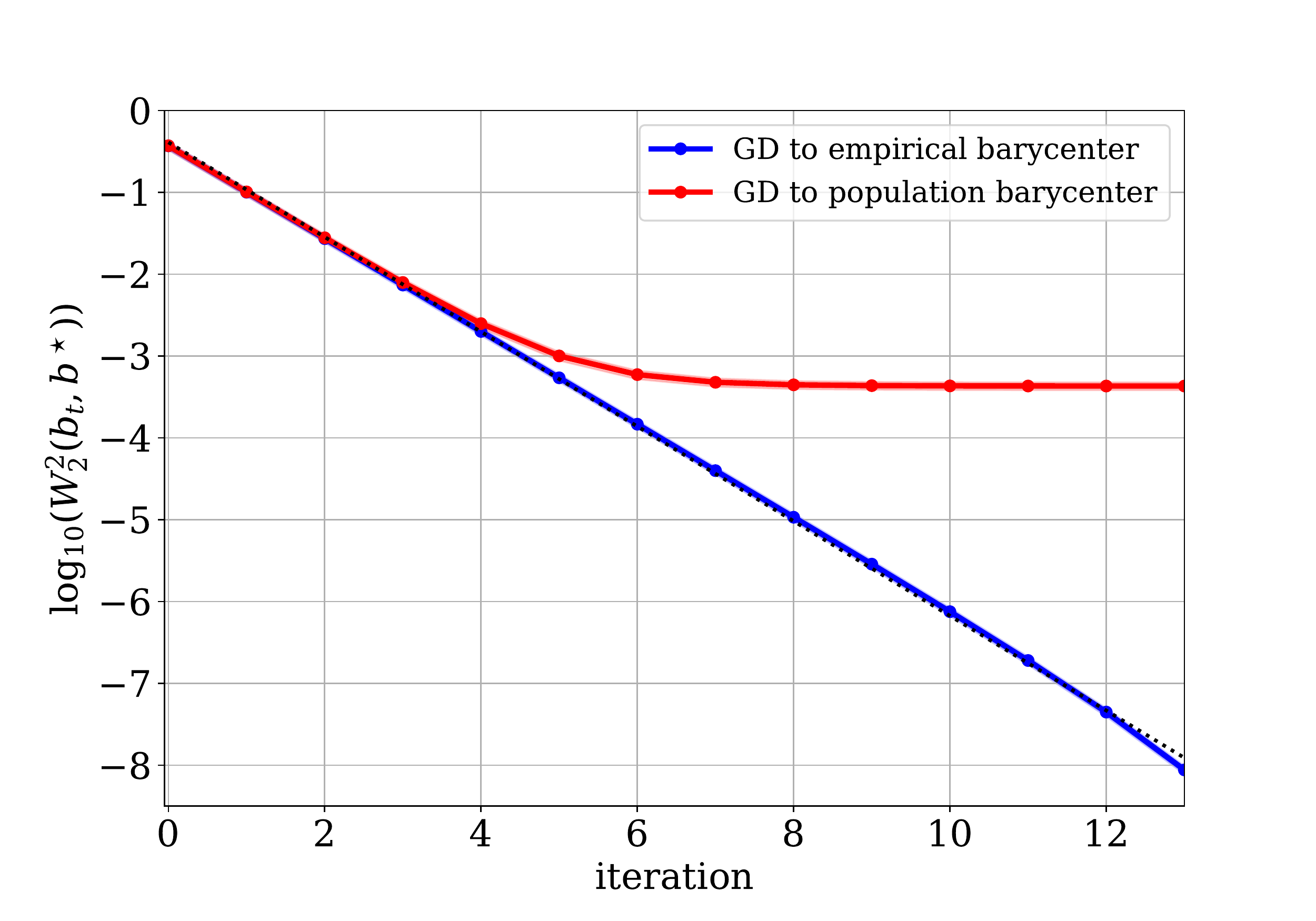}
    \caption{Left. Convergence of SGD on Bures manifold for $n=1000$, $d=3$, and and $b^\star=\gamma_{0,I_3}$.
   Right: linear convergence of GD on the same problem.}
    \label{fig:sgd}
\end{figure}

\section{Gradient descent on Wasserstein Space}\label{subsec:main_results}

In this section, we first review some background on optimal transport and describe first order algorithms on Wasserstein space. Then, we derive rates of convergence assuming a Polyak-\L{}ojasiewicz (PL) inequality. Theorems~\ref{thm:gd} and~\ref{thm:sgd} below are proved using modifications of the usual proofs in the optimization literature. Their proofs make critical use of the non-negative curvature of the Wasserstein space and are deferred to Appendix~\ref{appendix:rates}.

\subsection{Notation and background on optimal transport}
\label{SEC:OT}
In this section, we give a quick overview of the background and notation for optimal transport that is relevant for the paper. We provide a more thorough review of Riemannian geometry and the geometry of Wasserstein space in Appendix~\ref{appendix:geom}.
For each topic below, we also provide a reference to a useful presentation.

\noindent{\sf \underline{Wasserstein distance.}} \cite[Chapter 1]{villani2003topics}. Given a Polish space $(E,d)$, we denote by $\mc P_2(E)$ 
the collection of all (Borel) probability measures $\mu$ on $E$ such that 
$\mathbb{E}_{X\sim \mu}[{d(X,y)}^2] < \infty$ for some $y \in E$.
For two measures 
$\mu, \nu \in \mc P_2(E)$, let $\Pi_{\mu, \nu}$ be the set of couplings between $\mu$ and $\nu$, that is, the collection of probability measures
$\pi$ on $E\times E$ such that if $(X,Y) \sim \pi$, then
$X\sim\mu$ and $Y\sim \nu$. The $2$-Wasserstein distance between $\mu$ and $\nu$
is then defined as
\begin{equation}
    \label{EQ:defW2}
W_2^2(\mu, \nu) := \inf_{\pi \in \Pi_{\mu, \nu}}\mathbb{E}_{(X,Y) \sim \pi}[{d(X,Y)}^2].
\end{equation}
We are primarily interested in the cases $E = \R^D$ equipped with the standard Euclidean metric, and $E = \mc P_2(\R^D)$ equipped with the Wasserstein metric. Thus, $\mc P_2(\R^D)$ denotes the space of probability measures on $\R^D$ with finite second moment, and $\mc P_2(\mc P_2(\R^D))$ denotes the space of measures $P$ on $\mc P_2(\R^D)$ such that $\E_{\nu \sim P}W_2^2(\mu_0, \nu) < \infty$ for some, and therefore any, $\mu_0 \in \mc P_2(\R^D)$.
If $\mu \in \mc P_2(\R^D)$ is absolutely continuous w.r.t.\ the Lebesgue measure, we write $\mu \in \mc P_{2,\rm ac}(\R^D)$, and we similarly define the space $\mc P_2(\mc P_{2,\rm ac}(\R^D))$.

\noindent{\sf \underline{Transport map.}} \cite[Chapter 2]{villani2003topics}. Given a measure $\mu$ and a map $T \colon \R^D \to \R^D$, the pushforward
$T_\#\mu$ is the law of $T(X)$ when $X \sim \mu$.
For $\mu, \nu \in \mc P_{2,\rm ac}(\R^D)$, Brenier's theorem tells us that
there exists a unique optimal coupling $\pi^\star \in \Pi_{\mu,\nu}$ that achieves the minimum in~\eqref{EQ:defW2} and furthermore that it is induced by a mapping
$T_{\mu\to\nu}$, in the sense that if $X \sim \mu$ then $(X,T_{\mu\to\nu}(X)) \sim \pi^\star$. Moreover, $T_{\mu\to\nu}$ is the ($\mu$-a.e.\ unique) gradient of a convex function $\phi_{\mu\to\nu}$ such that \[ {(\nabla \phi_{\mu\to\nu})}_\# \mu = \nu. \]

\noindent{\sf \underline{Kantorovich potential.}} \cite[Chapter 2]{villani2003topics}. The $\phi_{\mu \to \nu}:\R^D \to \R$ specified in this way is called the \textit{Kantorovich potential} for the optimal transport from $\mu$ to $\nu$. For $\alpha, \beta>0$, if $\phi_{\mu\to\nu}$ is $\alpha$-strongly convex and $\beta$-smooth, in the sense that for all $x,y \in \R^D$,
\begin{align}\label{eq:sc_smooth}
    \frac{\alpha}{2}\|y-x\|^2 \le \phi_{\mu\to\nu}(y) - \phi_{\mu\to\nu}(x) - \langle \nabla \phi_{\mu\to\nu}(x), y-x \rangle \le \frac{\beta}{2}\|y-x\|^2,
\end{align}
then we say that the potential $\varphi_{\mu\to\nu}$ is \emph{$(\alpha,\beta)$-regular}.

\noindent{\sf \underline{Geodesics.}} \cite[Section 5.1]{villani2003topics}.
The space $\mc P_{2,\rm ac}(\R^D)$ space is a geodesic space, where the geodesics are given
by McCann's displacement interpolation.
Defining
$\mu_s := ((1 - s)\id + s T_{\mu_0 \to \mu_1})_{\#}\mu_0$, then ${(\mu_s)}_{s\in [0, 1]}$
is a constant-speed geodesic in Wasserstein space which connects $\mu_0$ to $\mu_1$. For any $\nu \in \mc P_{2,\rm ac}(\R^D)$, define the {\it generalized geodesic} with base $\nu$ and connecting $\mu_0$ to $\mu_1$  by ${(\mu_s^\nu)}_{s\in [0, 1]}$ where $\mu_s^\nu := {[(1-s) T_{\nu\to\mu_0} + s T_{\nu\to\mu_1}]}_\# \nu$.

%


\noindent{\sf \underline{Tangent bundle.}} \cite[Chapter 8]{ambrosio2008gradient}.
For $b \in \mc P_{2, \rm ac}(\R^D)$
define the ``tangent space" at $b$ by
$$
T_b \mathcal{P}_{2, \rm ac}(\R^D) := \overline{\left\{ \lambda ( \nabla \phi - \id)
\colon \lambda > 0, \; \phi \in C^{\infty}_{c}(\R^D), \; \textrm{$\phi$ convex}
\right\}}^{L^2(b)}.
$$ 
For $v \in T_b\mc P_{2, \rm ac}(\R^D)$ we write  $\|v\|_b:= \|v\|_{L^2(b)}$. Moreover, 
for any $b,b' \in \mc P_{2,\rm ac}(\R^D)$, define the map $\log_b: \mc P_{2,\rm ac}(\R^D) \to  T_b\mc P_{2, \rm ac} (\R^D)$ by $\log_b(b') := T_{b \to b'} - \id$. Reciprocally, we  define the map $\exp_b: T_b \mc P_{2,\rm ac}(\R^D) \to \mc P_{2,\rm ac}(\R^D)$ by $\exp_b(v)=({\id} + v)_{\#} b$.

\noindent{\sf \underline{Convexity.}} \cite[Section 7]{agueh2011barycenter}.
We are now in a position to define two notions of convexity in Wasserstein space.
Consider any functional $\mc F : \mc P_{2,\rm ac}(\R^D) \to (-\infty,\infty]$ on Wasserstein space.
We say that $\mc F$ is \emph{geodesically convex} if for all $\mu_0,\mu_1 \in \mc P_{2,\rm ac}(\R^D)$, the constant-speed geodesic ${(\mu_s)}_{s\in [0, 1]}$ from $\mu_0$ to $\mu_1$ satisfies $\mc F(\mu_s)
    \le (1-s) \mc F(\mu_0) + s \mc F(\mu_1)$ for all $s\in [0,1]$.
We say that $\mc F$ is \emph{convex along generalized geodesics} if for all choices $\nu, \mu_0, \mu_1 \in \mc P_{2,\rm ac}(\R^D)$, 
it holds that $\mc F(\mu_s^\nu)\le (1-s) \mc F(\mu_0) + s\mc F(\mu_1)$ for all $s\in [0,1]$.
Observe that the notion of generalized geodesic reduces to that of a geodesic when $\nu=\mu_0$, so that convexity along generalized geodesic is a stronger notion than convexity along geodesics. We say that a set $\mathcal{C}\subset  \mc P_{2,\rm ac}(\R^D)$ is convex along geodesics (resp. generalized geodesics) if its indicator function $\iota_{\mathcal{C}}$ is convex along geodesics (resp. generalized geodesics). Note that a set $\cC$ is convex along generalized geodesics with base $b$ if and only if the set $\log_b(\cC)$ is convex in the usual sense.

\noindent{\sf \underline{Curvature.}} \cite[Theorem 7.3.2]{ambrosio2008gradient}. Lastly, we  often use the fact that $\mc P_{2,\rm ac}(\R^D)$ is non-negatively curved in the sense of Alexandrov.
More specifically, we use the fact that for $\mu_0,\mu_1,\nu \in \mc P_{2,\rm ac}(\R^D)$, if ${(\mu_s)}_{s\in [0,1]}$ denotes the constant-speed geodesic connecting $\mu_0$ to $\mu_1$, then for all $s\in [0,1]$,
    \begin{align}\label{eq:nnc}
        W_2^2(\mu_s,\nu)
        &\ge (1-s)W_2^2(\mu_0,\nu) + sW_2^2(\mu_1,\nu) - s(1-s) W_2^2(\mu_0,\mu_1).
    \end{align}
Moreover, for any $\mu,\nu,b \in  \mc P_{2,\rm ac}(\R^D)$ the definition of Wasserstein distance implies
\begin{equation}\label{eq:distineq}
W_2(\mu, \nu) \le \norm{T_{b \to \nu} \circ T_{\mu \to b} - \id}_{L^2(\mu)}
= \norm{T_{b \to \nu} - T_{b \to \mu}}_{L^2(b)} = \norm{\log_b(\mu)- \log_b(\nu)}_b.
\end{equation}


\subsection{Gradient descent algorithms over Wasserstein space}

\subsubsection{Gradient descent.}
Let $Q$ be a probability distribution over $(\cP_{2,\rm ac}(\R^D), W_2)$. In the sequel, we focus on the cases where $Q=P$, $Q=P_n$, or $Q$ is a weighted atomic distribution,
but our results apply generically to any $Q$ that satisfy the conditions stated in the theorems below.

Using the techniques of~\cite{ambrosio2008gradient}, the \emph{gradient} of the barycenter functional $G$ defined in~\eqref{EQ:bary_functional}
may be easily computed~\cite{zemel2019procrustes}. Analogous to the Riemannian formula
(reviewed in Appendix~\ref{appendix:geom}), the Wasserstein gradient of $G$ at $b \in \mc P_{2,\rm ac}(\R^D)$ is the mapping $\nabla G(b) : \R^D\to\R^D$ defined by
\begin{align*}
    \nabla G(b) := -Q\log_b (\cdot) = -\int (T_{b\to\mu} - \id) \, \ud Q(\mu).
\end{align*}
Denote by $\bar b$ any minimizer of $G$. 

The primary assumption we work with is common in the optimization literature. 
We say that $G$ satisfies a \emph{Polyak-\L{}ojasiewicz (PL)} inequality  at $b$ if
\begin{align}\label{eq:pl}
    \|\gradG(b)\|_b^2 \ge 2C_{\msf{PL}} [G(b) - G(\bar b)]  \quad \text{for some }C_{\msf{PL}}>0.
\end{align}
It follows from~\eqref{eq:full_smoothness} below that $C_{\mathsf{PL}} \le 1$ for any such $Q$.
%

The \emph{gradient descent (GD)} iterates on $G$ are defined as
\begin{align}\label{eq:gd}
    b_0 \in \supp Q, \qquad b_{t+1} := \exp_{b_t}\bigl(-\gradG(b_t)\bigr)={[\id -\gradG(b_t)]}_\# b_t\quad\text{for}~t \ge 1.
\end{align}
Note that this method employs a unit step size. This is in agreement with the observation made in~\cite[Lemma 2]{zemel2019procrustes} that it leads to the best decrement in $G$ with respect to the smoothness upper bound, see Theorem~\ref{thm:gradient_drop} below.

The following theorem shows that
a PL inequality yields a linear rate of convergence.
\begin{thm}[Rate of convergence for gradient descent]\label{thm:gd}
        If $G$ satisfies the PL inequality~\eqref{eq:pl} at all the iterates ${(b_t)}_{t<T}$, then
        \begin{align*}
            G(b_T) - G(\bar b) \le {(1-C_{\msf{PL}})}^T [G(b_0) - G(\bar b)].
        \end{align*}
\end{thm}

\subsubsection{Stochastic gradient descent.}

PL inequalities are also useful in the stochastic setting where we observe $n$ independent copies $\mu_1, \ldots, \mu_n$ of   $\mu\sim Q$.
In this case, we consider the natural \emph{stochastic gradient descent (SGD)} iterates defined by
\begin{align}\label{eq:sgd}
    b_0 := \mu_0, \quad b_{t+1} := \exp_{b_t}\bigl(-\eta_t \log_{b_t}(\mu_{t+1})\bigr) = {[ \id + \eta_t( T_{b_t\to\mu_{t+1}} - \id)]}_\# b_t\quad\text{for}~t=0,\dotsc,n-1,
\end{align}
where $\eta_t \in (0,1)$ denotes the step size. 
At each iteration, SGD moves the iterate along the geodesic between $b_t$ and $\mu_{t+1}$ for step size $\eta_t$.
Under the assumption of a PL inequality, we show that SGD achieves a parametric rate of convergence.

In the following result, we recall that the \emph{variance} of $Q$ is defined as
\[ \var(Q) := \int W_2^2(\bar b, \cdot) \, \ud Q = 2G(\bar b). \]

\begin{thm}[Rates of convergence for SGD]\label{thm:sgd}
Assume that there exists a constant $C_{\msf{PL}}>0$ such that the following holds: $G$ satisfies the PL inequality~\eqref{eq:pl} at all the iterates ${(b_t)}_{0\le t\le n}$ of SGD run with step size  
\begin{equation}
\label{EQ:stepsize}
\eta_t = C_{\mathsf{PL}} \Bigl(1 - \sqrt{1- \frac{2(t+k)+1}{C_{\mathsf{PL}}^2 {(t+k+1)}^2}}\Bigr) 
\le \frac{2}{C_{\mathsf{PL}} (t+k+1)},
\end{equation}
where we take $k = 2/C_{\msf{PL}}^2 - 1\ge 0$.
Then,
\begin{align*}
    \E G(b_n) - G(\bar b)
    &\le \frac{3\var(Q)}{C_{\msf{PL}}^2 n}.
\end{align*}
\end{thm}
The parameter $k$ in~\eqref{EQ:stepsize} ensures that the step size is well-defined and less than 1.

\subsection{Properties of the barycenter functional}

Unlike results in generic optimization, this paper focuses on a specific function to optimize: the barycenter functional. In fact, this is a vast family of functionals, each indexed by the distribution $Q$ in~\eqref{EQ:bary_functional}. However, some structure is shared across this family. In the rest of this section, we extract properties that are relevant to our optimization questions: a variance inequality, smoothness, as well as an integrated PL inequality. These properties are valid for general distributions $Q$ over $\mc P_2(\R^D)$ and are specialized to the Bures manifold in the next section.
 
\subsubsection{Variance inequality.}\label{sscn:var_ineq}

Variance inequalities indicate quadratic growth of the barycenter functional around its minimum. 
More specifically, we say that $Q$ satisfies a \emph{variance inequality}   with constant $C_{\msf{var}} > 0$ if
\begin{align}\label{eq:var_ineq}
    G(b) - G(\bar b)
    \ge \frac{C_{\msf{var}}}{2} W_2^2(b, \bar b), \qquad \forall b \in \mc P_{2,\rm ac}(\R^D).
\end{align}
In particular,~\eqref{eq:var_ineq} implies uniqueness of $\bar b$.
The importance of variance inequalities for obtaining statistical rates of convergence for the empirical barycenter was emphasized in~\cite{ahidarcoutrix2018convergence}. In~\cite{legouic2019fast}, it is shown that an assumption on the regularity of the transport maps from the barycenter $\bar b$ implies a variance inequality.
Specifically, suppose that all of the Kantorovich potentials $\varphi_{\bar b\to\mu}$ for $\mu \in \supp Q$ are $(\alpha,\beta)$-regular in the sense of~\eqref{eq:sc_smooth}. Then, a variance inequality holds with $C_{\msf{var}} = 1 - (\beta - \alpha)$.

It turns out that a variance inequality holds without needing to assume smoothness of $\varphi_{\bar b \to \mu}$: assuming that the potential $\phi_{\bar b\to\mu}$ is $(\alpha(\mu),\infty)$-regular for each $\mu \in \supp Q$ yields a variance inequality with $C_{\msf{var}} = \int \alpha(\mu) \, \ud Q(\mu)$. The improvement here is critical for achieving global results on the Bures manifold. Moreover, when combined with the work of~\cite{ahidarcoutrix2018convergence} it yields improved statistical guarantees for the empirical barycenter.
To formally state this result, we need the notion of an \emph{optimal dual solution} for the barycenter problem. A discussion of this concept, along with a proof of the following theorem, is given in Appendix~\ref{appendix:proof_of_var}. We verify that the hypotheses of the theorem hold in the case when $Q$ is supported on non-degenerate Gaussian measures in Appendix~\ref{appendix:gaussreg}.

\begin{thm}[Variance inequality]\label{thm:variance_ineq}
Fix $Q \in \mc P_2(\mc P_{2,\rm ac}(\R^D))$ be a distribution with barycenter $\bar b \in \mc P_{2,\rm ac}(\R^D)$.
    Assume that there exists an optimal dual solution $\varphi$ for the barycenter problem w.r.t.\ $\bar b$ 
    such that, for $Q$-a.e.\ $\mu \in \mc P_{2,\rm ac}(\R^D)$, the mapping $\varphi_\mu$ is $\alpha(\mu)$-strongly convex for some measurable function $\alpha : \mc P_2(\R^D) \to \R_+$.
    Then, $Q$ satisfies a variance inequality~\eqref{eq:var_ineq} with constant $$C_{\msf{var}} = \int \alpha(\mu) \, \ud Q(\mu)\,.$$
\end{thm}


\subsubsection{Smoothness.}\label{subsec:smoothness}

Recall that a convex differentiable function $f : \R^D \to \R$ is $\beta$-smooth if
\begin{align}\label{eq:euclidean_smooth}
    f(y) \le f(x) + \langle \nabla f(x), y-x \rangle + \frac{\beta}{2} \|y-x\|^2, \qquad \forall x,y \in \R^D.
\end{align}
A consequence of $\beta$-smoothness is the following inequality, which measures how much progress gradient descent makes in a single step~\cite{bubeck2015convex}.
\begin{align}\label{eq:gradient_drop}
    f\big(x - \beta^{-1} \nabla f(x)\big) - f(x) \le - \frac{1}{2\beta} \|\nabla f(x)\|^2.
\end{align}
In fact, only the latter inequality~\eqref{eq:gradient_drop} is needed for the analysis of gradient descent methods. It was noted, first in~\cite[Proposition 3.3]{esteban2016barycenters} and then in~\cite[Lemma 2]{zemel2019procrustes}, that an analogue of~\eqref{eq:gradient_drop} holds in Wasserstein space for the barycenter functional.  Below, we provide a different, more geometric proof of this fact that emphasizes the collective role of smoothness and curvature. On the way, we also establish a smoothness inequality~\eqref{eq:full_smoothness} that is used in the proof of Theorem~\ref{thm:gd} and also ensures that $C_{\mathsf{PL}}\le 1$ for any distribution $Q$ supported on $ \mc P_{2,\rm ac}(\R^D)$.
%
%

\begin{thm}\label{thm:gradient_drop}
For any $b_0, b_1 \in \mc P_{2,\rm ac}(\R^D)$ the barycenter functional satisfies the smoothness inequality
    \begin{align}\label{eq:full_smoothness}
        G(b_1) \le G(b_0) + \langle \nabla G(b_0), \log_{b_0} b_1 \rangle_{b_0} + \frac{1}{2} W_2^2(b_0,b_1).
    \end{align}
Moreover, for any $b  \in \mc P_{2,\rm ac}(\R^D)$ and $b^+ := {[\id-\gradG(b)]}_\# b$, it holds.
    \begin{align}\label{eq:smoothness}
        G(b^+) - G(b) \le -\frac{1}{2} \norm{\gradG(b)}_b^2.
    \end{align}
\end{thm}

\begin{proof}
    Let ${(b_s)}_{s\in [0,1]}$ be the constant-speed geodesic between arbitrary $b_0, b_1 \in \mc P_{2, \rm ac}(\R^D)$.
    From the non-negative curvature inequality~\eqref{eq:nnc}, it holds that for any $s \in (0,1]$,
    \begin{align*}
        \int \frac{W_2^2(b_s,\mu) - W_2^2(b_0,\mu)}{s} \, \ud  Q(\mu)
        &\ge \int [W_2^2(b_1,\mu) - W_2^2(b_0,\mu)] \, \ud Q(\mu) - (1-s)W_2^2(b_0,b_1).
    \end{align*}
    By dominated convergence, the left-hand side converges to
    $$
    \int   \frac{\ud}{\ud s}W_2^2(b_s,\mu)\big|_{s=0_+}\, \ud  Q(\mu)= -2 \int   \langle  T_{b_0\to \mu}-\id,T_{b_0\to b_1} -\id\rangle_{L_2(b_0)}\, \ud  Q(\mu)=2\langle \nabla G(b_0), \log_{b_0}(b_1)\rangle_{b_0}\,,
    $$
    where in the first identity, we used the characterization of~\cite[Proposition 7.3.6]{ambrosio2008gradient}.
Rearranging terms yields~\eqref{eq:full_smoothness}.

Noticing that $W_2^2(b,b^+)=\norm{-\gradG(b)}_b^2$, Theorem~\ref{thm:gradient_drop} is now an immediate consequence of~\eqref{eq:full_smoothness} applied to $b_0 = b$ and $b_1 = b^+ $.
\end{proof}


\subsubsection{An integrated PL inequality.}

The main technical hurdle of this work is to provide sufficient conditions under which the PL inequality holds.
The following lemma, proved in Appendix~\ref{sscn:main_lemma}, is our main device to establish PL inequalities.

\begin{lem}\label{lem:main}
Let $Q$ satisfy a variance inequality with constant $C_{\msf{var}}$ and let $b\in\mc P_{2,\rm ac}(\R^D)$ be such that the barycenter $\bar b$ of $Q$ is absolutely continuous w.r.t.\ $b$.
Assume further the following measurability conditions: there exists a measurable mapping $\phi : \mc P_2(\R^D) \times \R^D \to \R \cup \{\infty\}$, $(\mu,x)\mapsto \phi_{b\to\mu}(x)$, such that, for $Q$-almost every  $\mu \in \mc P_{2,\rm ac}(\R^D)$, $\phi_{b\to\mu} : \R^D \to \R \cup \{\infty\}$ is a Kantorovich potential for the optimal transport from $b$ to $\mu$.
Then, \[ G(b) - G(\bar b) \le \frac{2}{C_{\msf{var}}} {\Bigl(\int_0^1 \norm{\gradG(b)}_{L^2(b_s)} \, \ud s\Bigr)}^2, \] where ${(b_s)}_{s\in [0,1]}$ is the constant-speed $W_2$-geodesic beginning at $b_0 := b$ and ending at $b_1 := \bar b$.
\end{lem}

%
%
%
This lemma can yield a PL inequality in quite general situations,
but the crucial issue is whether these conditions hold uniformly
for each iterate in the optimization trajectory. In the next section, we show how to turn an integrated PL inequality into a bona fide PL inequality when $Q$ is supported on certain Gaussian measures.
%

\section{Gradient descent on the Bures-Wasserstein manifold}\label{sec:bures_gradient_2descent}

Upon identifying a centered non-degenerate Gaussian measure with its covariance matrix, the Wasserstein geometry induces a Riemannian structure on the space of positive definite matrices, known as the \emph{Bures} geometry. Accordingly, we now refer to the barycenter of $Q$ as the \emph{Bures-Wasserstein barycenter}. 

\subsection{Bures-Wasserstein gradient descent algorithms}

We now specialize both GD and SGD when $Q$ is supported on mean-zero Gaussian measures. In this case, the updates of both algorithms take a remarkably simple form. To see this, for $m\in\R^D$, $\Sigma \in \mbb S_+^D$, let $\gamma_{m,\Sigma}$ denote the Gaussian measure on $\R^D$ with mean $m$ and covariance matrix $\Sigma$. The set of non-degenerate Gaussians constitutes
a well-behaved subset of Wasserstein space, called
the \emph{Bures-Wasserstein} manifold \cite{bures1969extension, bhatia2019bures}.
In particular, the
optimal coupling between $\gamma_{m_0, \Sigma_0}$ and
$\gamma_{m_1, \Sigma_1}$ has the explicit form
\begin{equation}\label{eq:gaussian_opt_coupling}
x \mapsto T_{\gamma_{\mu_0,\Sigma_0} \to \gamma_{\mu_1,\Sigma_1}}(x) := m_1 + \Sigma_0^{-1/2} {(\Sigma_0^{1/2}\Sigma_1 \Sigma_0^{1/2})}^{1/2}
\Sigma_0^{-1/2}(x - m_0).
\end{equation}
Observe that $T_{\gamma_{\mu_0,\Sigma_0} \to \gamma_{\mu_1,\Sigma_1}}$ is affine, and thus $\int T_{\gamma_{\mu_0,\Sigma_0} \to \gamma} \, \ud Q(\gamma)$ is affine.

This means that all of the GD (or SGD) iterates are Gaussian measures, so it suffices to keep track of the mean and covariance matrix of the current iterate. For both GD and SGD, the update equation for the descent step decomposes into two decoupled equations: an update equation for the mean, and an update equation for the covariance matrix. Moreover, the update equation for the mean is trivial, corresponding to a simple GD or SGD procedure on the objective function $m \mapsto \int \|m-m(\mu)\|^2 \, \ud Q(\mu)$, which is just mean estimation in $\R^D$.
Therefore, for simplicity and without loss of generality, we consider only mean-zero Gaussians throughout this paper and
we simply have to write down the update equations for the covariance matrix $\Sigma_t$ of the iterate.
The resulting update equations
are summarized in Algorithms~\ref{ALG:GD} and~\ref{ALG:SGD} below.




%
\begin{minipage}{0.48\textwidth}
\begin{algorithm}[H]
  \caption{Bures-Wasserstein GD}\label{ALG:GD}
  \begin{algorithmic}[1]
    \Procedure{Bures-GD}{$\Sigma_0, Q, T$}
      \For{$t = 1, \ldots, T$}
        \State $
        S_t \gets \int \Sigma_{t - 1}^{-1/2} {\{\Sigma_{t - 1}^{1/2} \Sigma(\mu) \Sigma_{t - 1}^{1/2}\}}^{1/2} \Sigma_{t - 1}^{-1/2} \, \ud Q(\mu)
        $
        \State $ \Sigma_{t} \gets S_t\Sigma_{t - 1} S_t$
      \EndFor
      \State \textbf{return } $\Sigma_{T}$
    \EndProcedure
  \end{algorithmic}
\end{algorithm}
\end{minipage}
\hfill
\begin{minipage}{0.48\textwidth}
\begin{algorithm}[H]
  \caption{Bures-Wasserstein SGD}\label{ALG:SGD}
  \begin{algorithmic}[1]
    \Procedure{Bures-SGD}{$\Sigma_0, {(\eta_t)}_{t = 1}^T, {(K_t)}_{t = 1}^T$}
      \For{$t = 1, \ldots, T$}
        \State $
        \hat{S}_t \gets \Sigma_{t - 1}^{-1/2} {\{\Sigma_{t - 1}^{1/2} K_t \Sigma_{t - 1}^{1/2}\}}^{1/2} \Sigma_{t - 1}^{-1/2}
        $
        \State $ \Sigma_{t} \gets ((1 - \eta_t)I_D + \eta_t
        \hat{S}_t)\Sigma_{t - 1} ((1 - \eta_t)I_D + \eta_t
        \hat{S}_t)$
      \EndFor
      \State \textbf{return } $\Sigma_{T}$
    \EndProcedure
  \end{algorithmic}
\end{algorithm}
\end{minipage}

\medskip

In the rest of this section, we prove the guarantees for GD and SGD on the Bures-Wasserstein manifold given in Theorems~\ref{thm:GD_short} and~\ref{thm:SGD_short}.
\subsection{Proof of the main results}

For simplicity, we make the following reductions: we assume that the Gaussians are centered (see previous subsection) and that the eigenvalues of the covariance matrices of the Gaussians are uniformly bounded above by $1$. The latter assumption is justified by the observation that if there is a uniform upper bound on the eigenvalues of the covariance matrices, then we can apply a simple rescaling argument (Lemma~\ref{lem:rescaling} in the Appendix).

While the centering and scaling assumptions stated above can be made without loss of generality, our results require the following regularity condition. Note that it is equivalent to a uniform upper bound on the densities of the Gaussians.

\begin{defi}[$\sep$-regular]
Fix $\sep \in (0,1]$. A distribution $Q \in \mathcal{P}_2(\R^D)$ is said to be $\sep$-regular if its support is contained in 
\begin{equation}
\label{eq:S_eta}
 \cS_\sep=\big\{\gamma_{0,\Sigma} \,:\, \Sigma \in \mbb S_{++}^D, \; \|\Sigma\|_{\rm op} \le 1, \; \det \Sigma \ge \sep\big\}.
\end{equation}
\end{defi}
Hereafter, we always assume that $Q$ is $\sep$-regular for some $\sep>0$. Under this condition, it can be shown that the barycenter of $Q$ exists and is unique (Proposition~\ref{prop:gaussian_barycenter} in the Appendix).

We begin with a brief outline of the proof.
\begin{itemize}
    \item[(i)] If we initialize gradient descent (or stochastic gradient descent) at one of the elements of the support of $Q$, then all of the iterates, all of the elements of $\supp Q$, the barycenter $\bar b$, and all of elements of geodesics between these measures are  non-denegerate Gaussians $\gamma_{0,\Sigma} \in \cS_\sep$.
    \item[(ii)] Using Lemma~\ref{lem:main}, we establish a PL inequality holds with a uniform constant over $ \cS_\sep$.
    \item[(iii)] The guarantees for GD and SGD on the Bures manifold follow immediately from the PL inequality and our general convergence results (Theorems~\ref{thm:gd},~\ref{thm:sgd}).
\end{itemize}

In the sequel, we use \emph{geodesic convexity} as a key tool to control the iterates of the gradient descent algorithm. We note that this discussion is not about proving some sort of geodesic convexity for our objective, which cannot hold in general. Our main interest in geodesic convexity comes from the following fact: if all of the elements of the support of $Q$ lie in a geodesically convex set $\cS_\sep$, and we initialize the algorithm at an element of $\cS_\sep$, then all of the iterates of stochastic gradient descent are
simply moving along geodesics within this set, and so remain in $\cS_\sep$. The same is true for the iterates of gradient descent, provided that we replace geodesic convexity with \emph{convexity along generalized geodesics}. Refer to Section~\ref{SEC:OT} for definitions
of these terms.
We begin with the following fact.

\begin{lem}\label{thm:geod_cvx_max_eigenvalue}
    For a measure $\mu \in \mc P_2(\R^D)$, let $M(\mu) := \int x\otimes x \, \ud \mu(x)$.
    Then, the functional $\mu \mapsto \|M(\mu)\|_{\rm op} = \lambda_{\max}(M(\mu))$ is convex along generalized geodesics on $\mc P_2(\R^D)$.
\end{lem}
\begin{proof}
  Let $S^{D-1}$ denote the unit sphere of $\R^D$ and observe that for any $e \in S^{D-1}$ the function $x\mapsto \langle x, e \rangle^2$ is convex on $\R^D$.
    By known results for geodesic convexity in Wasserstein space (see~\cite[Proposition 9.3.2]{ambrosio2008gradient}), the functional $\mu \mapsto \int \langle \cdot, e \rangle^2 \, \ud \mu = \langle e, M(\mu) e \rangle$ is convex along generalized geodesics in $\mc P_2(\R^D)$; hence, so is the functional $\mu \mapsto \max_{e \in S^{D-1}} \langle e, M(\mu) e \rangle = \|M(\mu)\|_{\rm op}$.
\end{proof}

The next lemma establishes convexity along generalized geodesics of $\mu \mapsto -\ln \det \Sigma(\mu)$. It follows from specializing Lemma~\ref{lem:geod_cvx_log_density} in the Appendix to the Bures-Wasserstein manifold.
\begin{lem}\label{thm:geod_cvx_log_det}
    The functional $\gamma_{0,\Sigma} \mapsto -\sum_{i=1}^D \ln \lambda_i(\Sigma)$ is convex along generalized geodesics on the space of non-degenerate Gaussian measures.
\end{lem}

It follows readily from Lemmas~\ref{thm:geod_cvx_max_eigenvalue} and~\ref{thm:geod_cvx_log_det} that the set $\cS_\sep$ is convex along generalized geodesics. Moreover  since SGD moves along geodesics and is initialized at $b_0 \in \supp Q \subset \cS_\sep$, then all the iterates of SGD stay in $\cS_\sep$. To show that the same holds for GD, 
observe that the set $\log_{b_t}(\cS_\sep)$ is convex. Therefore, $-\gradG(b_t)=\int (T_{b_t \to \mu} -\id)\, \ud Q(\mu) \in \log_{b_t}(\cS_\sep)$ as a convex combination of elements in this set. This is equivalent to $b_{t+1}=\exp_{b_t}(-\gradG(b_t)) \in \cS_\sep$. These observations yield the following corollary.

\begin{cor}
The set $\cS_\sep$ is  convex along generalized geodesics and when initialized in $\supp Q$, the iterates of both GD and SGD remain in $\cS_\sep$.
\end{cor}

This completes the first step (i) of the proof. Moving on to step (ii), we get from Theorem~\ref{thm:pl_bures} that $G$ satisfies a PL inequality with constant $C_{\mathsf{PL}}=\zeta^2/4$ at all $b \in \cS_\zeta$ and in particular at all the iterates of both GD and SGD.

Combined with the general bound in Theorems~\ref{thm:gd} and the variance inequality in Theorem~\ref{thm:varineq_bures}, this completes the proof of Theorems~\ref{thm:GD_short} for GD.
To prove Theorem~\ref{thm:SGD_short}, take $k=1/C_{\mathsf{PL}}=4/\zeta^2$ so that Theorem~\ref{thm:sgd} yields
$$
 \E G(b_n) - G(\bar b) \le  \frac{48 \var(Q)}{n \zeta^4}\,.
$$
Combining this bound with the variance inequality in Theorem~\ref{thm:varineq_bures} completes the proof of Theorem~\ref{thm:SGD_short}.

\bigskip

\noindent{\bf Acknowledgments.} \\ PR was supported by NSF awards IIS-1838071, DMS-1712596, {DMS-TRIPODS-1740751}, and ONR grant N00014-17- 1-2147. 
Sinho Chewi and Austin J. Stromme were supported by the Department of Defense (DoD) through the National Defense Science \& Engineering Graduate Fellowship (NDSEG) Program. We thank
 the anonymous reviewers for
helpful comments.

\appendix

\section{Geometry and Wasserstein space}\label{appendix:geom}
In this section we give a more detailed introduction to Riemannian
manifolds and discuss analogies to Wasserstein space which are
present throughout the paper. We refer readers to~\cite{docarmo1992riemannian} for a standard introduction to Riemannian geometry.

\subsection{Riemannian geometry}

An $n$-dimensional manifold $M$ is a topological space
which is Hausdorff, second countable,
and locally homeomorphic to $\R^n$. A smooth atlas
is a collection of smooth charts
${\{\psi_{\alpha}\}}_{\alpha \in \mathcal{A}}$ so that each $\psi_{\alpha} \colon U_{\alpha} \subset M \to \R^n$
is a homeomorphism from an open set $U_{\alpha}$ in $M$,
$M = \bigcup_{\alpha \in \mathcal{A}} U_{\alpha}$,
and such that for all $\alpha, \alpha' \in \mathcal{A}$,
$\psi_{\alpha} \circ \psi_{\alpha'}^{-1}$
is smooth wherever defined.
For a fixed choice of smooth atlas, we declare a function
$f \colon M \to \R$ to be smooth if $f \circ \psi_{\alpha}^{-1}$
is for each $\alpha \in \mathcal{A}$. The manifold together with a smooth atlas defines
a smooth $n$-dimensional manifold, and we shall always suppress
mention of the atlas. A map $f \colon M \to N$ between
two smooth manifolds is said to be smooth if its composition
with smooth charts is.

Given a smooth $n$-dimensional manifold $M$ and a point $p \in M$,
the tangent space $T_pM$ is the equivalence class of all smooth
curves $\gamma \colon (-\eps, \eps) \to M$ such that $\gamma(0) = p$,
where two such curves $\gamma_0, \gamma_1$ are equivalent
if, with respect to every coordinate chart $\psi$ defined in a neighborhood of $p$,
${(\psi \circ \gamma_0)}'(0) = {(\psi \circ \gamma_1)}'(0)$.
As such, $T_pM$ is a real $n$-dimensional
vector space for each $p \in M$. The cotangent space at $p\in M$
is then the dual to $T_pM$, which we shall denote $T_p^* M$.
The tangent bundle is the
disjoint union $TM := \bigsqcup_{p \in M} T_pM$,
and the cotangent bundle is similarly the disjoint union
$T^*M := \bigsqcup_{p \in M} T_p^*M$. The smooth
structure on $M$ induces a smooth structure on $TM$ and $T^*M$,
so each is then a $2n$-dimensional smooth manifold in its own right.

A smooth vector field $X \colon M \to TM$ is then a smooth map $p\mapsto X_p$
such that $X_p \in T_pM$ for all $p \in M$, and similarly
for a smooth covector field $\alpha \colon M \to T^*M$.
Higher-order tensors are defined similarly: a $(p,q)$-tensor field is a smooth mapping $T \colon M \to {(TM)}^p \otimes {(T^* M)}^q$.
The differential $df \colon M \to T^*M$ of a smooth function $f$ on $M$ is the smooth covector field
such that $df_p \colon T_pM \to \R$
obeys $df_p(v) := (f\circ \gamma)'(0)$, where $\gamma$ is any curve
with tangent vector $v \in T_pM$ at $\gamma(0) = p$.

A Riemannian manifold $(M,g)$ is a smooth $n$-dimensional manifold $M$
with a smooth metric tensor $g \colon M \to T^*M \otimes T^*M$; at each point of $M$, this is a positive definite bilinear form.
The metric tensor therefore defines a smoothly varying choice of inner product
on the tangent spaces of $M$. In addition to giving rise to notions of length and geodesics, the metric tensor provides a canonical isomorphism (the Riesz isomorphism) between the tangent space and cotangent space: for a vector $v \in T_pM$
the covector $\alpha_v \in T_p^*M$ is defined by $\alpha_v(w) = g_p(v,w)$.
For a covector $\alpha \in T_p^*M$ the vector $v_{\alpha} \in T_pM$ is defined
as the unique solution of $\alpha(w) = g_p(v_{\alpha}, w)$ for all $w \in T_pM$.
A smooth vector field $X$ can be accordingly transformed into a smooth
covector field denoted $X^{\flat}$, and a smooth covector field $\omega$ can
be transformed into a smooth vector field $\omega^{\#}$.
The gradient of a function $f \colon M \to \R$ is defined then as $\nabla f:=(df)^{\#}$: in other words, for all $p \in M$ and $v \in T_p M$, $df_p(v) = g_p(\nabla f(p), v)$.

We typically write $\langle \cdot, \cdot \rangle_p$ instead of $g_p(\cdot, \cdot)$, and we write $\norm\cdot_p$ for the norm induced by the metric tensor, i.e., $\norm v_p := \sqrt{\langle v, v\rangle_p}$.
In this notation, the distance between points $p,q \in M$ is defined as
$$
d_M(p,q):= \inf_{\gamma \in \Gamma(p,q)} \int_0^1 \|\gamma'(t)\|_{\gamma(t)} \, \ud t,
$$ where $\Gamma(p,q)$ is the collection of all smooth (or piecewise continuous)
curves $\gamma \colon [0, 1] \to M$ 
such that $\gamma(0) = p$ and $\gamma(1) = q$. If $M$ is connected, then
the distance $d_M$ is indeed a metric.
If we additionally assume that $(M, d_M)$ is complete as a metric space then
by the Hopf-Rinow theorem the value of the above minimization problem
is attained by at least one curve 
$\gamma \colon [0, 1] \to M$ such that $t\mapsto \norm{\gamma'(t)}_{\gamma(t)}$ is constant, which is said to be a constant-speed (minimizing)
geodesic.

For any $p \in M$, there always exists an $\eps > 0$ such that for any vector $v \in T_pM$  with $\norm v_p < \eps$, there is a unique constant-speed geodesic
$\gamma_{v} \colon [0,1] \to M$ obeying $\gamma_v(0) = p$ and $\gamma_v'(0) = v$.\footnote{In fact, a stronger result holds: there exists a neighborhood $U$ of $p$ such that for any two points $q, q' \in U$, there is a unique constant-speed minimizing geodesic $\gamma \colon [0,1] \to U$ joining $q$ to $q'$. Such a neighborhood is called a totally normal neighborhood of $p$.}
On the ball $B_\eps(0)$ with radius $\eps$ and center $0\in T_p M$ (with respect to the norm $\norm \cdot_p$), we can now define the exponential map $\exp_p \colon B_\eps(0) \to M$ by $v \in V_p \mapsto \gamma_v(1)$.
The exponential map is a diffeomorphism onto its image, so we can define the inverse mapping $\log_p \colon \exp_p(B_\eps(0)) \to T_p M$.
If $M$ is complete, the domain of definition of any constant-speed geodesic $\gamma \colon [0,1]\to\R$ can be extended to all of $\R$ such that at each time $\gamma$ is locally a constant-speed minimizing geodesic; in this case, the exponential mapping can be extended to a mapping $\exp_p \colon T_p M\to M$.
Note, however, that the mapping $\log_p$ is not necessarily defined everywhere.

We lastly recall that for fixed $q \in M$ and $p$ which does not belong to the cut locus of $q$ (the set of points for which there exists more than one constant-speed minimizing geodesic from $p$),
$$
[\nabla d^2_M(\cdot, q)](p) = - 2\log_p(q).
$$ This statement has an intuitive meaning: it simply says that outside
of the cut locus of $q$, the gradient of the squared distance
points in the direction of maximum increase.\footnote{When there are multiple constant-speed minimizing geodesics joining $p$ to $q$, then the following fact is still true: the squared distance function $d^2_M(\cdot, q)$ is superdifferentiable at $p$. Moreover, for any constant-speed minimizing geodesic $\gamma \colon [0,1] \to M$ joining $p$ to $q$, the vector $-2\gamma'(0) \in T_p M$ is a supergradient of $d^2_M(\cdot, q)$ at $p$.}

\subsection{Riemannian interpretation of Wasserstein space}

In this section, we briefly explain the interpretation set out in~\cite{otto2001geometry} of the Wasserstein space of probability measures as a Riemannian manifold. For more introductory expositions of this subject, we refer to~\cite[Chapter 8]{villani2003topics} and~\cite[Chapter 5]{San15}. The task of putting this formal discussion on rigorous footing is undertaken in~\cite[Chapter 8]{ambrosio2008gradient}. We also note that many treatments view Wasserstein space as a length space using the framework of metric geometry; see~\cite{buragoivanov2001metricgeometry} for an introduction to this approach.

Let $\mu_0 \in \mc P_{2,\rm ac}(\R^D)$ and consider a family ${(v_t)}_{t\in [0, 1]}$ of smooth vector fields on $\R^D$, that is, $v_t : \R^D\to\R^D$ for each $t\in [0, 1]$. Suppose we draw $X_0 \sim \mu_0$ and we evolve $X_0$ according to the ODE $\dot X_t = v_t(X_t)$ for $t\in [0,1]$, that is, we seek an integral curve of ${(v_t)}_{t\in [0,1]}$ with starting point $X_0$. If we let $\mu_t$ denote the law of $X_t$, we may compute the evolution of ${(\mu_t)}_{t\in [0,1]}$ as follows. Take any smooth test function $\psi$ on $\R^D$, and (ignoring any issues of regularity) compute
\begin{align*}
    \partial_t \int \psi \, \ud \mu_t
    &= \partial_t \E\psi(X_t)
    = \E \partial_t \psi(X_t)
    = \E \langle \nabla \psi(X_t), v_t(X_t) \rangle
    = \int \langle \nabla \psi, v_t \rangle \, \ud \mu_t
    = -\int \psi \Div(v_t \mu_t).
\end{align*}
This suggests that the pair ${(\mu_t)}_{t\in [0,1]}$, ${(v_t)}_{t\in [0,1]}$ should solve the following PDE, which is known as the continuity equation:
\begin{align}\label{eq:continuity}
    \partial_t \mu_t
    + \Div(v_t \mu_t) = 0.
\end{align}
This PDE can be interpreted in a suitable weak sense, e.g.: for any smooth test function $\psi$ with compact support, the mapping $t\mapsto \int \psi \, \ud \mu_t$ should be absolutely continuous and thus differentiable at almost every $t\in [0,1]$, and its derivative should satisfy $\partial_t \int \psi \, \ud \mu_t = \int \langle \nabla \psi, v_t \rangle \, \ud \mu_t$.

Since the vector fields ${(v_t)}_{t\in [0,1]}$ govern the evolution of the curve ${(\mu_t)}_{t\in [0,1]} \subseteq \mc P_{2,\rm ac}(\R^D)$, we would like to equip $\mc P_{2,\rm ac}(\R^D)$ with the structure of a Riemannian manifold such that ${(v_t)}_{t\in [0,1]}$ is interpreted as the tangent vectors to the curve ${(\mu_t)}_{t\in [0,1]}$. However, a problem arises: given a curve ${(\mu_t)}_{t\in [0,1]}$ in Wasserstein space, there are many choices for the vector fields ${(v_t)}_{t\in [0,1]}$ which solve~\eqref{eq:continuity} together with ${(\mu_t)}_{t\in [0,1]}$. Indeed, if we fix any pair ${(\mu_t)}_{t\in [0,1]}$, ${(v_t)}_{t\in [0,1]}$ solving~\eqref{eq:continuity}, then  we obtain another solution by replacing $v_t$ with $v_t + w_t$, where $w_t$ is any vector field satisfying $\Div(w_t \mu_t) = 0$. So, what should we take as the tangent vectors to ${(\mu_t)}_{t\in [0,1]}$?

We can take a hint from optimal transport. Specifically, Brenier's theorem asserts that in the optimal transport problem of transporting a measure $\nu_0$ to another measure $\nu_1$, the optimal transport plan is induced by a transport map, which is the gradient of a convex function $\phi$. In other words, if we interpret $\nu_0$ as a collection of particles, then each particle initially moves along the vector field $\nabla \phi - \id$. In particular, taking $\nu_0 = \mu_0$ and $\nu_1 = \mu_\eps$ for a small $\eps > 0$, we expect the tangent vector of ${(\mu_t)}_{t\in [0,1]}$ at time $0$ to be of the form $\nabla \phi - \id$ for a convex function $\phi$.

This motivates the definition of the tangent space to $\mc P_{2,\rm ac}(\R^D)$ at a measure $b$, given in~\cite[Chapter 8]{ambrosio2008gradient} as
$$
T_b \mathcal{P}_{2, \rm ac}(\R^D) := \overline{\left\{ \lambda \, ( \nabla \phi - \id)
\colon \lambda > 0, \; \phi \in C^{\infty}_{c}(\R^D), \; \textrm{$\phi$ convex}
\right\}}^{L^2(b)},
$$ where the closure is with respect to the $L^2(b)$ distance.
We equip $T_b \mc P_{2,\rm ac}(\R^D)$ with the $L^2(b)$ metric, that is, for vector fields $v, w \in T_b \mc P_{2,\rm ac}(\R^D)$ we define $\langle v, w \rangle_b := \langle v, w \rangle_{L^2(b)} := \int \langle v, w \rangle \, \ud b$. The metric induced by this Riemannian structure recovers the Wasserstein distance, in the sense that
\begin{align*}
    W_2(\mu_0,\mu_1)
    &= \inf\Bigl\{ \int_0^1 \norm{v_t}_{\mu_t} \, \ud t \Bigm\vert {(\mu_t)}_{t\in [0,1]}, {(v_t)}_{t\in [0,1]}~\text{solves}~\eqref{eq:continuity}\Bigr\}.
\end{align*}

Given two measures $\mu_0, \mu_1 \in \mc P_{2,\rm ac}(\R^D)$, there is a unique constant-speed minimizing geodesic joining $\mu_0$ to $\mu_1$.
It is given by $\mu_t := {[(1-t) \, {\id} + t T]}_\# \mu_0$, where $T$ is the optimal transport mapping from $\mu_0$ to $\mu_1$; this is known as McCann's interpolation. It satisfies
\begin{align}\label{eq:const_speed_geod}
    W_2(\mu_0, \mu_t) = t W_2(\mu_0, \mu_1) \qquad \forall t \in [0, 1].
\end{align}
Moreover, it can be shown that any constant-speed geodesic in $\mc P_{2,\rm ac}(\R^D)$, that is, any curve ${(\mu_t)}_{t\in [0,1]} \subseteq \mc P_{2,\rm ac}(\R^D)$ satisfying~\eqref{eq:const_speed_geod}, is necessarily of the form $\mu_t = {[(1-t) \, {\id} + tT]}_\# \mu_0$. The tangent vector to ${(\mu_t)}_{t\in [0, 1]}$ at time $0$ is the vector field $T - {\id}$.

Given $\mu \in \mc P_{2,\rm ac}(\R^D)$ and $v \in T_\mu \mc P_{2,\rm ac}(\R^D)$, we may now define the exponential map to be $\exp_\mu v := {({\id} + v)}_\# \mu$. Given any other $\nu \in \mc P_{2,\rm ac}(\R^D)$, we also define the logarithmic map to be $\log_\mu \nu := T_{\mu\to\nu} - {\id}$, where $T_{\mu\to\nu}$ is the optimal transport map from $\mu$ to $\nu$. Observe that $\log_\mu \nu$ is well-defined for \emph{any} pair $\mu,\nu \in \mc P_{2,\rm ac}(\R^D)$.

\section{Experiments}
\label{sec:experiments}

In this section, we demonstrate the linear convergence of GD, the fast rate of estimation for SGD, and some potential advantages of averaging stochastic gradient by way of numerical experiments. In evaluating SGD, we also include a variant that involves sampling with replacement from the empirical distribution.

\subsection{Simulations for the Bures manifold}

First, we begin by illustrating how SGD indeed achieves the fast rate of convergence to the true barycenter on the Bures manifold, as indicated by Theorem~\ref{thm:SGD_short}. 

To generate distributions with a known barycenter, we use the following fact.
If the mean of the distribution ${(\log_{b^\star})}_\# P$ is $0$, then $b^\star$ is a barycenter of $P$.
This fact follows from our PL inequality (Theorem~\ref{thm:pl_bures}) or also from general arguments in~\cite[Theorem 2]{zemel2019procrustes}.
We also use the fact that the tangent space of the Bures manifold is given by the set of all symmetric matrices~\cite{bhatia2019bures}.

Figure~\ref{fig:sgd} shows convergence of SGD for distributions on the Bures manifold. 
To generate a sample, we let $A_i$ be a matrix with i.i.d.\ $\gamma_{0,\sigma^2}$ entries. Our random sample on the Bures manifold is then given by
\begin{equation}\label{eq:buresdistribution}
    \Sigma_i = \exp_{\gamma_{0, I_D}} \Bigl( \frac{A_i + A_i^\top}{2} \Bigr),
\end{equation}
which has population barycenter $b^\star = \gamma_{0,I_D}$. An explicit form of this exponential map is derived in~\cite{malago2018wasserstein}. We run two versions of SGD. The first variant uses each sample only once, and passes over the data once. The second variant samples from $\Sigma_1, \dots, \Sigma_n$ with replacement at each iteration and takes the stochastic gradient step towards the selected matrix. For the resulting sequences, we also show the results of averaging the iterates. Specifically, if ${(b_t)}_{t \in \mathbb{N}}$ is the sequence generated by SGD, then the averaged sequence is given by $\tilde{b}_0 = b_0$ and
\[
    \tilde{b}_{t+1} = {\Bigl[ \frac{t}{t+1} \id +  \frac{1}{t+1} T_{\tilde{b}_t \to b_{t+1}} \Bigr]}_{\#} \tilde{b}_t.
\]
On Riemannian manifolds, averaged SGD is known to attain optimal statistical rates under smoothness and geodesic convexity assumptions~\cite{tripuraneni2018averaging}.

Here, we generate 100 datasets of size $n=1000$ in the way specified above and set $\sigma^2 = 0.25$. In this experiment, the SGD step size is chosen to be $\eta_t = 2 / [0.7 \cdot(t + 2/0.7 + 1)]$. The results from these 100 datasets are then averaged for each algorithm, and we also display 95\% confidence bands for the resulting sequences. As is clear from the log-log plot in Figure~\ref{fig:sgdloglog}, SGD achieves the fast $O(n^{-1})$ statistical rate on this dataset.


The right of Figure~\ref{fig:sgd} shows convergence of GD to the empirical barycenter and true barycenter. We generate samples in the same way as before. This linear convergence was observed previously by~\cite{esteban2016barycenters}.

\begin{figure}[t]
    \centering
    \includegraphics[width = .49\textwidth]{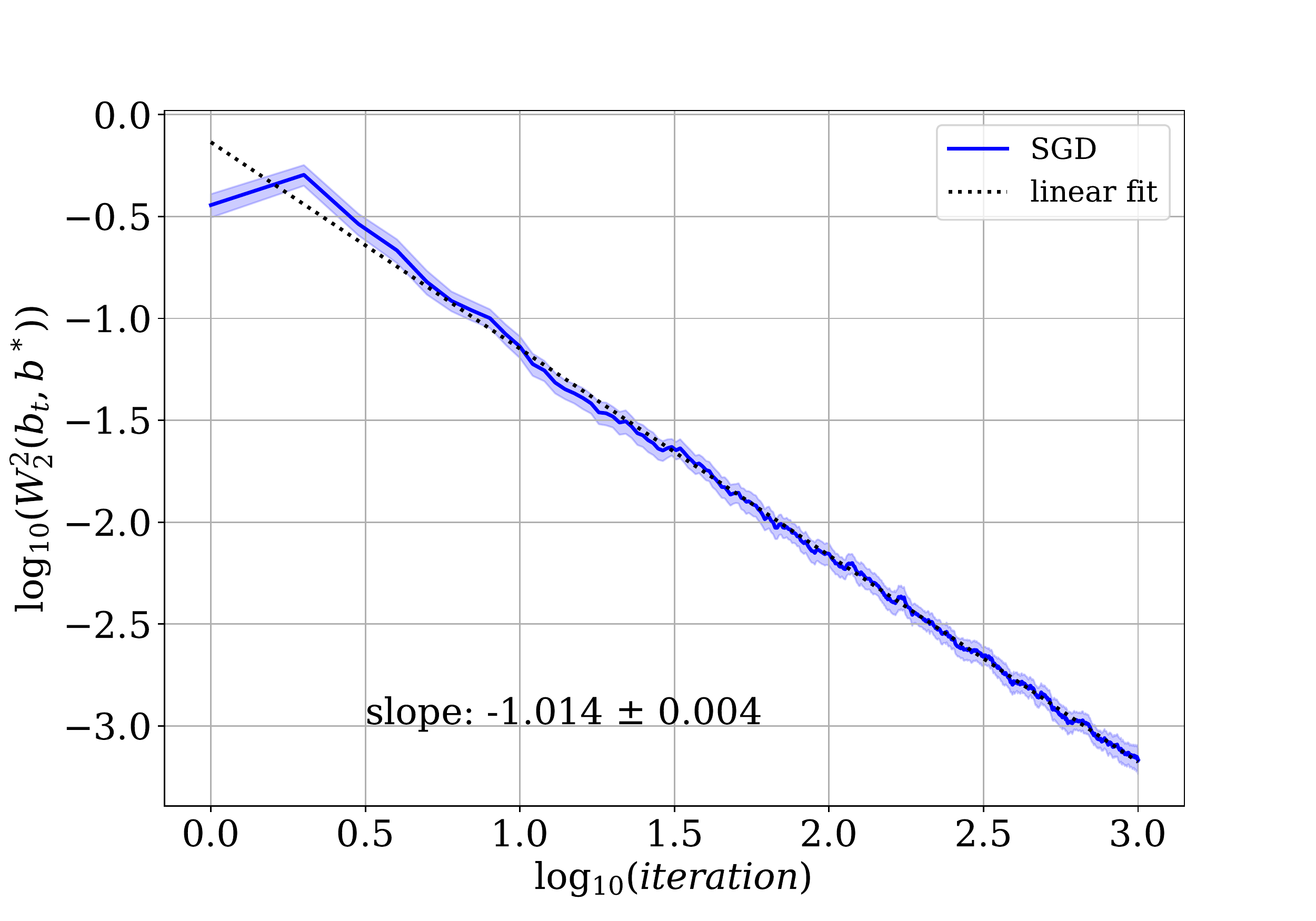}
    \caption{Log-log plot of convergence for SGD on Bures manifold for $n=1000$, $d=3$, and and $b^\star=\gamma_{0,I_3}$. This corresponds to the experiment on the left in Figure~\ref{fig:sgd}}
    \label{fig:sgdloglog}
\end{figure}

\begin{figure}[t]
    \centering
    \includegraphics[width = .49\textwidth]{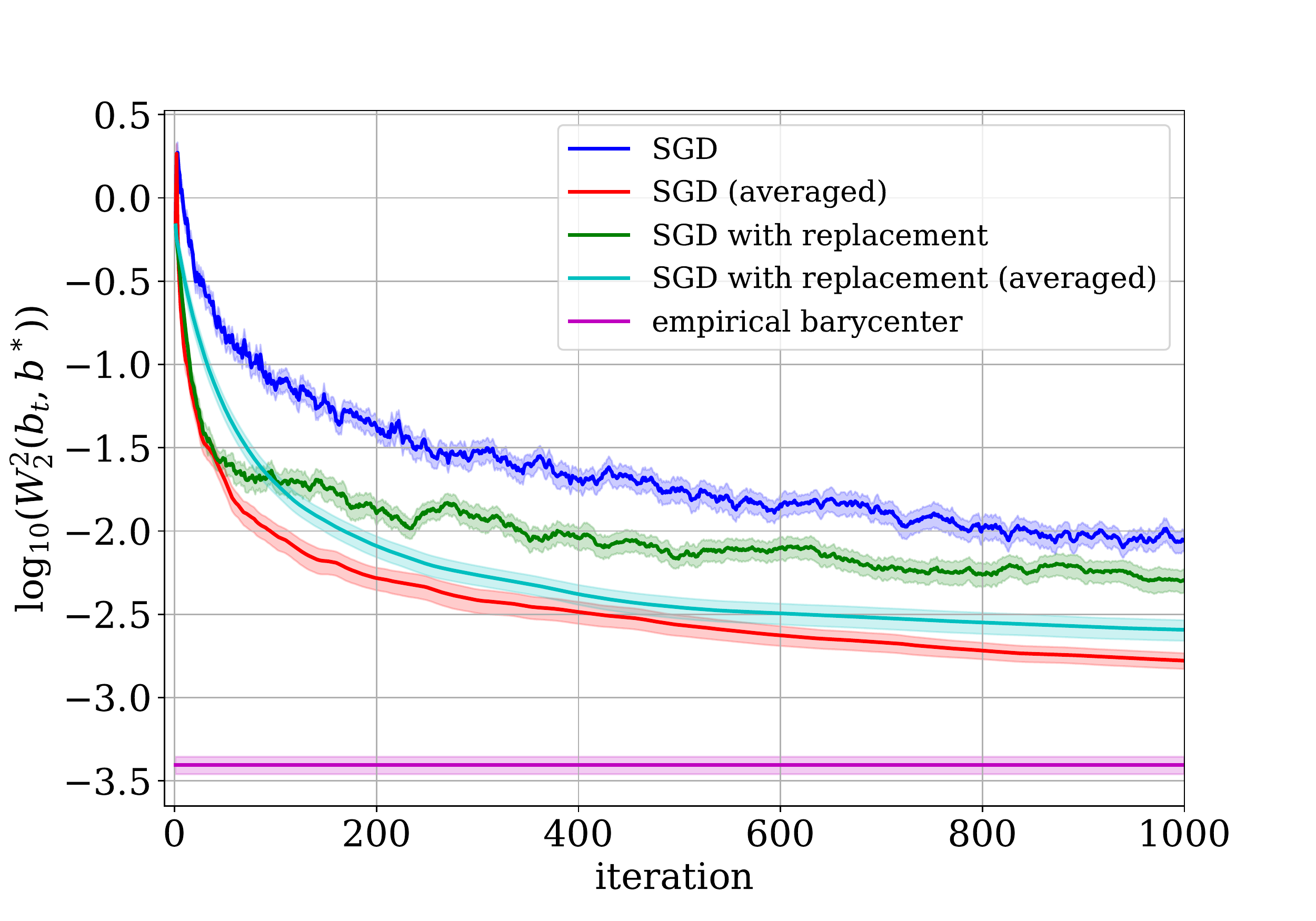}
    \includegraphics[width = .49\textwidth]{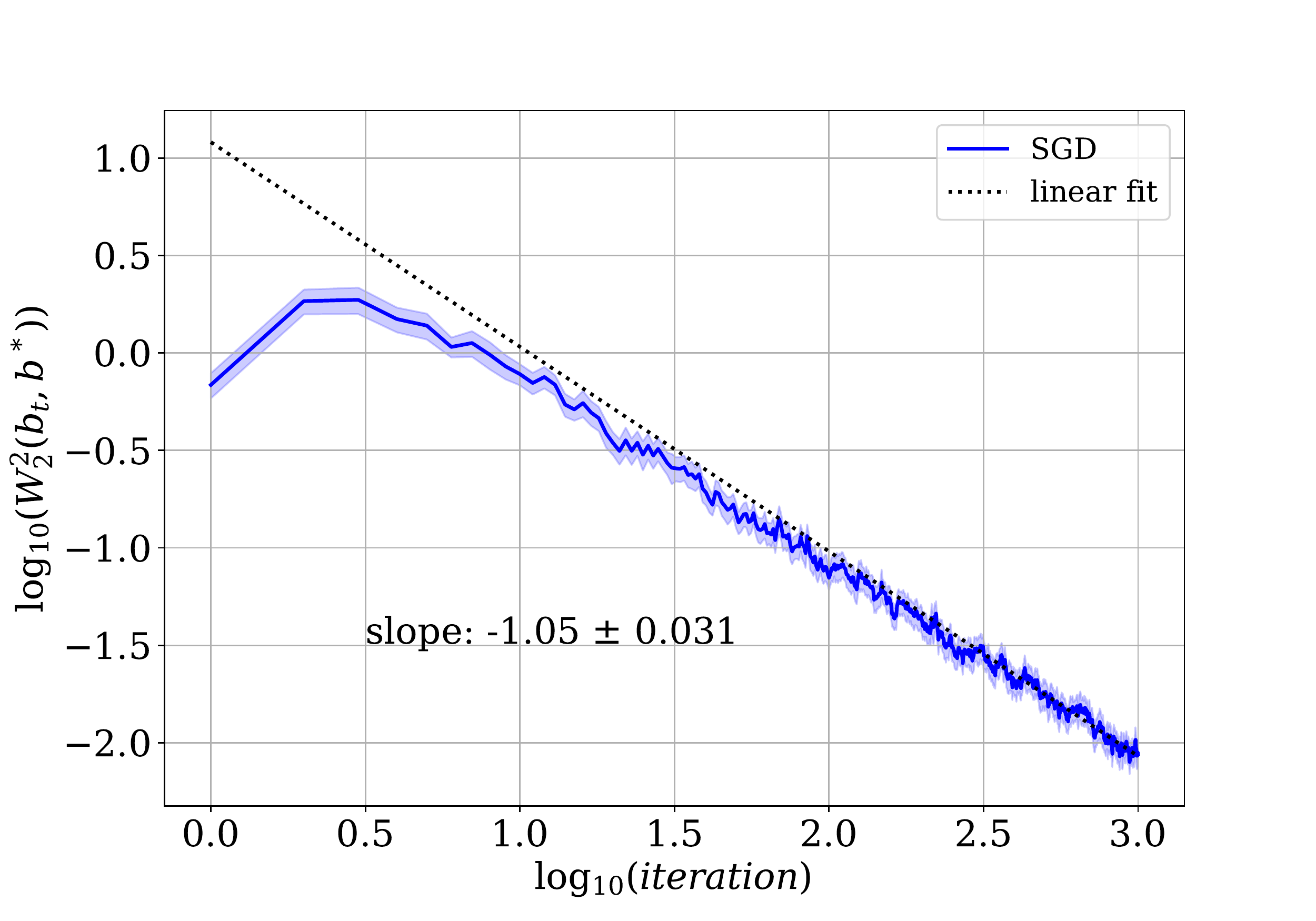}
    \caption{Convergence of SGD on Bures manifold. Here, $n=1000$, $d=3$, and barycenter given by $\on{diag}(20,1,1)$.
    The result displays the average over 100 randomly generated datasets.}
    \label{fig:sgd_poorcond}
\end{figure}

In Figure~\ref{fig:sgd_poorcond}, we repeat the same experiment, except this time the barycenter has covariance matrix
\begin{equation*}
    \Sigma^\star = \begin{pmatrix}
        20 & 0 & 0 \\
        0 & 1 & 0 \\
        0 & 0 & 1 \\
    \end{pmatrix}
    ,
\end{equation*}
and the entries of $A_i$ are drawn i.i.d.\ from $\gamma_{0,1}$.
In this situation, the condition numbers of the matrices generated according to this distribution are typically much larger than those centered around $\gamma_{0, I_3}$.
To account for a potentially smaller PL constant, we chose $\eta_t = 2 / [0.1 \cdot (t + 2/0.1 + 1)]$. It is again clear from the right pane in Figure~\ref{fig:sgd_poorcond} that SGD achieves the fast $O(n^{-1})$ statistical rate on this dataset. To account for the slow convergence initially, we only fit this line to the last 500 iterations. We also note that averaging yields drastically better performance in this case, which we are currently unable to theoretically justify.

Figure~\ref{fig:sgdr} shows convergence of SGD with replacement to the empirical barycenter. We generate $n=500$ samples in the same way as in Figure~\ref{fig:sgd}, where the true barycenter is $I_3$ and $\sigma^2 = 0.25$. We calculate the error obtained by the empirical barycenter by running GD on this dataset until convergence, which is displayed with the green line. We also calculate the error obtained by a single pass of SGD, which is given by the blue line. SGD with replacement is then run for 5000 iterations, and we observe that it does indeed achieve better error than single pass SGD if run for long enough. SGD with replacement converges to the empirical barycenter, albeit at a slow rate.

\begin{figure}[t]
    \centering
    \includegraphics[width = .49\textwidth]{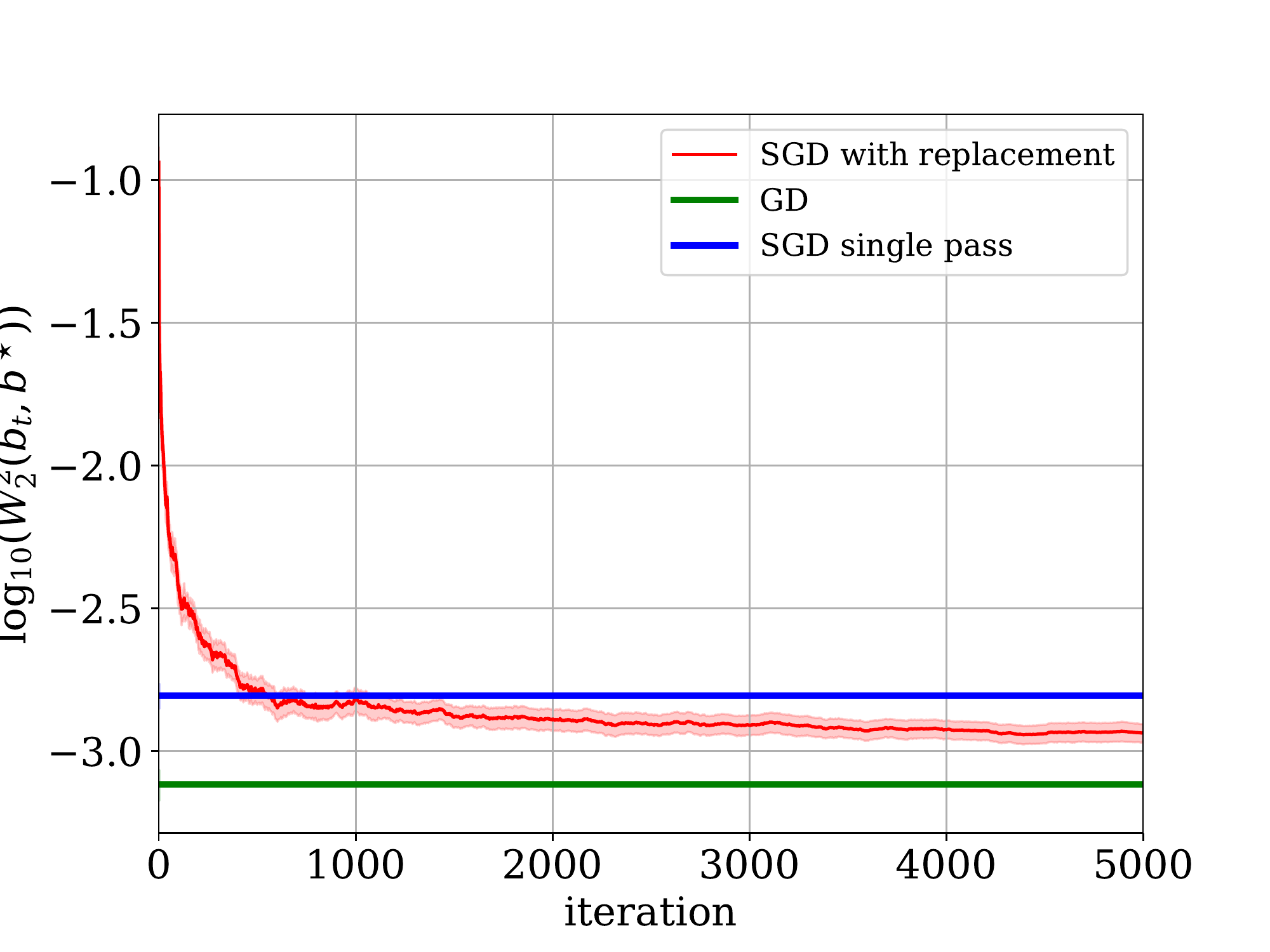}
    \caption{Convergence of SGD on Bures manifold. Here, $n=500$, $d=3$, and the distribution is given by \eqref{eq:buresdistribution} with $\Sigma^\star = I_3$ and $\sigma^2 = 0.25$. The result displays the average over 100 randomly generated datasets.}
    \label{fig:sgdr}
\end{figure}

\subsection{Details of the non-convexity example}\label{appendix:fig_details}

We consider the example of the Wasserstein metric restricted
to centered Gaussian measures, which induces the Bures metric on positive
definite matrices.
Even restricted to such Gaussian measures, the Wasserstein barycenter objective is geodesically non-convex, despite the fact that it is Euclidean convex~\cite{weber2019nonconvex}.
Figure~\ref{fig:buresnoncvx} gives a simulated example of this fact. 
This figure plots the Bures distance squared between a positive definite matrix $C$ and points along some geodesic $\gamma$, which runs between two matrices $A$ and $B$. The matrices used in this example are
\begin{equation*}
    A = \begin{pmatrix}
        0.8 & -0.4 \\
        -0.4 &  0.3
    \end{pmatrix}, \qquad B = \begin{pmatrix}
        0.3 & -0.5 \\
        -0.5 & 1.0
    \end{pmatrix}, \qquad C = \begin{pmatrix} 
        0.5 & 0.5 \\
        0.5 & 0.6
 \end{pmatrix},
\end{equation*}
and $\gamma(t)$, $t \in [0,1]$, is taken to be the Bures or Euclidean geodesic from $A$ to $B$ (the Euclidean geodesic is given by $t\mapsto (1-t)A + t B$). This function is clearly non-convex, and therefore we cannot assume that there is some underlying strong convexity (although the Bures distance is in fact strongly geodesically convex for sufficiently small balls~\cite{huang2015broyden}).

\section{Omitted proofs}\label{appendix:rates}

\subsection{Convergence bounds for GD and SGD under a PL inequality}

This subsection gives proofs of the general convergence theorems for GD and SGD in the present paper. Both of these proofs use the non-negative curvature inequality~\eqref{eq:distineq}. We note that the proof of Theorem~\ref{thm:gd} uses the non-negative curvature implicitly by invoking smoothness, while the use of non-negative curvature is explicit within the proof of Theorem~\ref{thm:sgd}. 
%

\subsubsection{Proof of Theorem~\ref{thm:gd} for GD.}
    
    Using the smoothness~\eqref{eq:smoothness} and the PL inequality~\eqref{eq:pl}, it holds that
    \begin{align*}
        G(b_{t+1}) - G(b_t)
        &\le - C_{\msf{PL}} [G(b_t) - G(\bar b)].
    \end{align*}
    It yields $G(b_{t+1}) - G(\bar b) \le (1-C_{\msf{PL}}) [G(b_t) - G(\bar b)]$, which gives the result.

\subsubsection{Proof of Theorem~\ref{thm:sgd} for SGD.}
Recall the SGD iterations on $n+1$ observations:
	\begin{align*}
    b_0 := \mu_0, \qquad b_{t+1} := {[(1-\eta_t) \id + \eta_t T_{b_t\to\mu_{t+1}}]}_\# b_t\quad\text{for}~t=0,\dotsc,n,
    \end{align*}
where the step size is given by
\[
\eta_t = C_{\mathsf{PL}} \Bigl(1 - \sqrt{1- \frac{2(t+k)+1}{C_{\mathsf{PL}}^2 {(t+k+1)}^2}}\Bigr) 
\le \frac{2}{C_{\mathsf{PL}} (t+k+1)},
\]
for some $k$ such that $C_{\msf{PL}}^2(k+1)^2 \ge 2k+1$. We note that the step size $\eta_t$ is chosen to solve the equation \[ 1-2C_{\mathsf{PL}} \eta_t + \eta_t^2 = {\Bigl(\frac{t+k}{t+k+1} \Bigr)}^2. \]
	
Using the non-negative curvature~\eqref{eq:distineq}, we get
	\begin{align*}
	W_2^2(b_{t+1}, \mu)
	&\le\norm{\log_{b_t} b_{t+1} - \log_{b_t}\mu}_{b_t}^2
	= \norm{\eta_t \log_{b_t} \mu_{t+1} - \log_{b_t} \mu}_{b_t}^2 \\ \nonumber
	&= \norm{\log_{b_t} \mu}_{b_t}^2 + \eta_t^2 \norm{\log_{b_t} \mu_{t+1}}_{b_t}^2 -  2 \eta_t \langle \log_{b_t} \mu, \log_{b_t} \mu_{t+1} \rangle_{b_t}.
	\end{align*}
	Taking the expectation with respect to $(\mu, \mu_{t+1}) \sim Q^{\otimes 2}$ (conditioning appropriately on the increasing sequence of $\sigma$-fields), we have
	\begin{align*}
	\E G(b_{t+1}) \le \E[(1+\eta_t^2)G(b_t) - \eta_t \|\nabla G(b_t)\|_{L^2(b_t)}^2].
	\end{align*}
	
	Using the PL inequality~\eqref{eq:pl}, 
    \begin{align*}
	\E G(b_{t+1})
	    &\le \E\big[(1+\eta_t^2) G(b_t) - 2C_{\mathsf{PL}} \eta_t [G(b_t) - G(\bar b)]\big].
    \end{align*}
	Subtracting $G(\bar b)$ and rearranging,
	\begin{align*}
	\E G(b_{t+1}) - G(\bar b)
	&\le (1 - 2C_{\mathsf{PL}} \eta_t + \eta_t^2)[\E G(b_t) -  G(\bar b)] + \frac{\eta_t^2}{2} \var(Q),
	\end{align*}
	where we recall that $\var(Q)=2G(\bar b)$.
	With the chosen step size, we find
    \begin{align*}
        \E G(b_{t+1}) - G(\bar b) &\le {\Bigl(\frac{t+k}{t+k+1} \Bigr)}^2 [ \E G(b_t) -  G(\bar b) ] + \frac{2 \var(Q)}{C_{\mathsf{PL}}^2 {(t+k+1)}^2}.
    \end{align*}
    Or equivalently,
    \begin{align*}
        {(t+k+1)}^2[\E G(b_{t+1}) - G(\bar b)]
        &\le {(t+k)}^2   [ \E G(b_t) -  G(\bar b)] + \frac{2\var(Q)}{C_{\msf{PL}}^2}.   
    \end{align*}
   Unrolling over $t=0,1,\dots,n-1$ yields
    \begin{align*}
        {(n+k)}^2[\E G(b_{n}) - G(\bar b)] \le  k^2 [\E G(b_0) - G(\bar b)] + \frac{2n\var(Q)}{C_{\mathsf{PL}}^2},
    \end{align*}
    or, equivalently,
\begin{equation}
\label{EQ:prSGD1}
  \E G(b_n) - G(\bar b) \le \frac{k^2}{{(n+k)}^2} [\E G(b_0) - G(\bar b)] + \frac{2\var(Q)}{C_{\mathsf{PL}}^2 (n+k)}.
\end{equation}
%
    To conclude the proof, recall that from~\eqref{eq:full_smoothness}, we have
    $$
G(b_0) - G(\bar b) \leqslant \frac{1}{2}
W_2^2(b_0, \bar b).
$$ Taking the expectation over $b_0 \sim Q$
we find
$$
\E G(b_0) - G(\bar b) \leqslant G(\bar b)
= \frac{1}{2} \var(Q),
$$ as claimed.
Together with~\eqref{EQ:prSGD1}, it yields
$$
\E G(b_n) - G(\bar b) \le \frac{\var(Q)}{n+k}\Big(\frac{k^2}{2(n+k)}+\frac{2}{C_{\mathsf{PL}}^2} \Big)\le \frac{\var(Q)}{n}\Big(\frac{k+1}{2}+\frac{2}{C_{\mathsf{PL}}^2} \Big).
$$
Plugging-in the value of $k$ completes the proof.

\subsection{Variance inequality: Theorem~\ref{thm:variance_ineq}} \label{appendix:proof_of_var}

We begin this section with a review of Kantorovich duality, which we use to discuss the dual of the barycenter problem. Then, we present the proof of Theorem~\ref{thm:variance_ineq}.

Given two measures $\mu,\nu \in \mc P_2(\R^D)$ and maps $f \in L^1(\mu)$, $g \in L^1(\nu)$ such that $f(x) + g(y) \ge \langle x, y \rangle$ for $\mu$-a.e.\ $x\in\R^D$ and $\nu$-a.e.\ $y\in\R^D$, it is easy to see that
\begin{align*}
    \frac{1}{2} W_2^2(\mu,\nu) \ge \int \Bigl(\frac{\norm \cdot^2}{2} - f\Bigr) \, \ud \mu + \int \Bigl(\frac{\norm \cdot^2}{2} - g\Bigr) \, \ud \nu.
\end{align*}
Kantorovich duality (see e.g.~\cite{villani2003topics}) says that equality holds for some pair $f = \varphi$, $g = \varphi^*$ where $\varphi$ is a proper LSC convex function and $\varphi^*$ denotes its convex conjugate, i.e.,
\begin{align*}
    \frac{1}{2} W_2^2(\mu,\nu) = \int \Bigl(\frac{\norm \cdot^2}{2} - \varphi\Bigr) \, \ud \mu + \int \Bigl(\frac{\norm \cdot^2}{2} - \varphi^*\Bigr) \, \ud \nu.
\end{align*}
The map $\varphi$ is called a Kantorovich potential for $(\mu,\nu)$.

Accordingly, given $\bar b \in \mc P_2(\R^D)$, we call a measurable mapping $\varphi : \mc P_{2,\rm ac}(\R^D) \to L^1(\bar b)$, $\mu \mapsto \varphi_\mu$, an \emph{optimal dual solution} for the barycenter problem if the following two conditions are met: (1) for $Q$-a.e.\ $\mu$, the mapping $\varphi_\mu$ is a Kantorovich potential for $(\bar b, \mu)$; (2) it holds that
\begin{align}\label{eq:opt_dual_barycenter}
    \int \Bigl(\frac{\norm\cdot^2}{2} - \varphi_\mu\Bigr) \, \ud Q(\mu) = 0.
\end{align}
It is easily seen that these conditions imply that $\bar b$ is the barycenter of $Q$:
\begin{align*}
    G(b)
    &= \frac{1}{2} \int W_2^2(b, \cdot) \, \ud Q
    \ge \int \Bigl[\int \Bigl(\frac{\norm \cdot^2}{2} - \varphi_\mu\Bigr) \, \ud b + \int \Bigl(\frac{\norm \cdot^2}{2} - \varphi_\mu^*\Bigr) \, \ud \mu \Bigr] \, \ud Q(\mu) \\
    &= \iint \Bigl(\frac{\norm \cdot^2}{2} - \varphi_\mu^*\Bigr) \, \ud \mu \, \ud Q(\mu)
    = \frac{1}{2} \int W_2^2(\bar b, \cdot) \, \ud Q
    = G(\bar b).
\end{align*}
The existence of an optimal dual solution for the barycenter problem is known in the finitely supported case~\cite{agueh2011barycenter}, and existence can be shown for the general case under mild conditions~\cite{legouic2020barycenter}. For completeness, we give a self-contained proof of the existence of an optimal dual solution in the case where $Q$ is supported on Gaussian measures in Appendix~\ref{appendix:gaussreg}.
\begin{proof}[Proof of Theorem~\ref{thm:variance_ineq}]
    By the strong convexity assumption, it holds for $Q$-a.e.\ $\mu \in \mc P_{2,\rm ac}(\R^D)$ and a.e.\ $x \in\R^D$,
    \begin{align*}
        \varphi_\mu^*(x) + \varphi_\mu(y)
        &\ge \langle x, y \rangle + \frac{\alpha(\mu)}{2} \norm{y-\nabla \varphi_\mu^*(x)}^2\,,
    \end{align*}
which can be rearranged into
    \begin{align*}
        \begin{aligned}
            &\norm{x-y}^2 - \alpha(\mu) \norm{y-\nabla \varphi_\mu^*(x)}^2
            \ge \frac{\norm x^2}{2} - \varphi_\mu^*(x) + \frac{\norm y^2}{2} - \varphi_\mu(y).
        \end{aligned}
    \end{align*}
    Integrating this w.r.t.  the optimal transport plan $\gamma_\mu$ between $\mu$ and $b \in \mc P_2(\R^D)$, yields
    \begin{align*}
        \frac{1}{2}\Bigl( W_2^2(\mu, b) - \alpha(\mu) \int \norm{T_{\mu\to b} - T_{\mu\to \bar b}}^2 \, \ud \mu \Bigr)
        &\ge \int \Bigl( \frac{\norm \cdot^2}{2} - \varphi_\mu^* \Bigr) \, \ud \mu + \int \Bigl( \frac{\norm \cdot^2}{2} - \varphi_\mu \Bigr) \, \ud b.
    \end{align*}
    Observe also that~\eqref{eq:distineq} implies $\norm{T_{\mu\to b} - T_{\mu\to \bar b}}^2_{L^2(\mu)} \ge W_2^2(b, \bar b)$.
   Integrating these inequalities with respect to $Q$ yields
    \begin{align*}
        G(b) - \frac{1}{2} \Bigl(\int \alpha \, \ud Q\Bigr) W_2^2(b, \bar b)
        &\ge \int \Bigl[ \int \Bigl( \frac{\norm \cdot^2}{2} - \varphi_\mu^*\Bigr) \, \ud \mu + \int \Bigl( \frac{\norm \cdot^2}{2} - \varphi_\mu \Bigr) \, \ud b \Bigr] \, \ud Q(\mu) \\
        &= \iint \Bigl( \frac{\norm \cdot^2}{2} - \varphi_\mu^* \Bigr) \, \ud \mu \, \ud Q(\mu)
        = G(\bar b).
    \end{align*}
    where in the last two identities, we used~\eqref{eq:opt_dual_barycenter}. It implies the variance inequality.
\end{proof}

\subsection{Integrated PL inequality}
%
\label{sscn:main_lemma}
The following lemma appears in~\cite[Lemma A.1]{lott2009ricci} in the case of Lipschitz functions. A minor modification of their proof allows to handle locally Lipschitz rather than only Lipschitz functions. We include the modified proof for completeness.

\begin{lem}\label{lem:lott_villani}
Let ${(b_s)}_{s\in [0,1]}$ be a Wasserstein geodesic in $\mc P_2(\R^D)$. Let $\Omega \subseteq \R^D$ be a convex open subset for which $b_0(\Omega) = b_1(\Omega) = 1$.
Then, for any function $f : \R^D \to \R$ which is locally Lipschitz on $\Omega$, it holds that
\begin{align*}
    \Bigl\lvert \int f \, \ud b_0 - \int f \, \ud b_1 \Bigr\rvert
    \le W_2(b_0, b_1) \int_0^1 \norm{\nabla f}_{L^2(b_s)} \, \ud s.
\end{align*}
\end{lem}
\begin{proof}
   According to~\cite[Corollary 7.22]{villani2009optimal}, there exists a probability measure $\Pi$ on the space of constant-speed geodesics in $\R^D$  such that $\gamma \sim \Pi$ and $b_s$ is the law of $\gamma(s)$. In particular, it yields
    \begin{align*}
        \int f \, \ud b_0 - \int f \, \ud b_1
        &= \int \bigl[f\bigl(\gamma(0)\bigr) - f\bigl(\gamma(1)\bigr)\bigr] \, \ud \Pi(\gamma).
    \end{align*}
    We can cover the geodesic ${(\gamma(s))}_{s\in [0,1]}$ by finitely many open neighborhoods contained in $\Omega$ so that $f$ is Lipschitz on each such neighborhood; thus, the mapping $t\mapsto f(\gamma(s))$ is Lipschitz and we may apply the fundamental theorem of calculus, the Fubini-Tonelli theorem, and Cauchy-Schwarz:
    \begin{align*}
        \int f \, \ud b_0 - \int f \, \ud b_1
        &= \int \int_0^1 \bigl\langle \nabla f\bigl(\gamma(s)\bigr), \dot\gamma(s) \bigr\rangle \, \ud s \, \ud \Pi(\gamma) \\
        &\le \int_0^1 \int \on{length}(\gamma) \bigl\lVert \nabla f\bigl(\gamma(s)\bigr) \bigr\rVert \, \ud \Pi(\gamma) \, \ud s \\
        &\le \int_0^1 \big(\int {\on{length}(\gamma)}^2 \, \ud \Pi(\gamma)\big)^{1/2} \big(\int \bigl\lVert \nabla f\bigl(\gamma(s)\bigr) \bigr\rVert^2 \, \ud \Pi(\gamma)\big)^{1/2} \, \ud s \\
        &= W_2(b_0,b_1) \int_0^1 \|\nabla f\|_{L^2(b_s)} \, \ud s.
    \end{align*}
    It yields the result.
\end{proof}

\begin{proof}[Proof of Lemma~\ref{lem:main}]
    By Kantorovich duality~\cite{villani2003topics},
    \begin{align*}
        \frac{1}{2} W_2^2(b, \mu)
        &= \int \Bigl(\frac{\norm\cdot^2}{2} - \phi_{\mu\to b}\Bigr) \, \ud \mu + \int \Bigl( \frac{\norm \cdot^2}{2} - \phi_{b\to\mu}\Bigr) \, \ud b, \\
        \frac{1}{2} W_2^2(\bar b, \mu)
        &\ge \int \Bigl(\frac{\norm\cdot^2}{2} - \phi_{\mu\to b}\Bigr) \, \ud \mu + \int \Bigl( \frac{\norm \cdot^2}{2} - \phi_{b\to\mu}\Bigr) \, \ud \bar b.
    \end{align*}
    This yields the inequality
    \begin{align*}
        G(b) - G(\bar b)
        &\le \int \Bigl( \frac{\norm \cdot^2}{2} - \int \phi_{b\to\mu} \, \ud Q(\mu) \Bigr) \, \ud (b-\bar b).
    \end{align*}
    Let $\bar\phi := \int \phi_{b\to\mu} \, \ud Q(\mu)$; this is a proper LSC convex function $\R^D\to \R \cup \{\infty\}$.
    We apply Lemma~\ref{lem:lott_villani} with $\Omega = \interior \dom \bar\phi$. Since $\bar\phi$ is locally Lipschitz on the interior of its domain (\cite[Theorem 10.4]{rockafellar1997convexanalysis} or~\cite[Theorem 4.1.3]{borwein2006convex}) and $\bar b \ll b$, then $b(\Omega) = \bar b(\Omega) = 1$, whence
    \begin{align*}
       G(b) - G(\bar b) &\le W_2(b,\bar b) \int_0^1 \norm{\nabla \bar \phi - \id}_{L^2(b_s)} \, \ud s
        \le \sqrt{\frac{2[G(b) - G(\bar b)]}{C_{\msf{var}}}} \int_0^1 \norm{\nabla \bar \phi - \id}_{L^2(b_s)} \, \ud s.
    \end{align*}
    Square and rearrange to yield
    \begin{align*}
        G(b) - G(\bar b) & \le \frac{2}{C_{\msf{var}}} \Bigl(\int_0^1 \norm{\nabla \overline{\phi} - \id}_{L^2(b_s)}\Bigr)^2 \, \ud s.
    \end{align*}
    Recognizing that $\gradG(b) = \id - \nabla \overline{\phi}$ yields the
    result.
\end{proof}

\subsection{Rescaling lemma}
\begin{lem}\label{lem:rescaling}
    For any $\alpha > 0$ and $\mu \in \mc P_2(\R^D)$, let $\mu_\alpha$ be the law of $\alpha X$, where $X \sim \mu$.
    Let $\mu \sim Q$ be a random measure drawn from $Q$, and let $Q_\alpha$ be the law of $\mu_\alpha$.
    Then, $\bar b$ is a barycenter of $Q$ if and only if $\bar b_\alpha$ is a barycenter of $Q_\alpha$.
\end{lem}
\begin{proof}
    It is an easy calculation to see that for any $\mu,\nu\in\mc P_2(\R^D)$, \[ W_2(\mu_\alpha,\nu_\alpha) = \alpha W_2(\mu,\nu) \] (see, for instance,~\cite[Proposition 7.16]{villani2003topics}).
    Let \[ G_\alpha(b) := \frac{1}{2} \int W_2^2(\cdot, b) \, \ud Q_\alpha(\mu). \]
    By the previous reasoning, $G_\alpha(b_\alpha) = \alpha^2 G(b)$.
    In particular, the mapping $\bar b \mapsto \bar b_\alpha$ is a one-to-one correspondence between the minimizers of these two functionals.
\end{proof}

\subsection{Properties of the Bures-Wasserstein barycenter}\label{appendix:gaussreg}

Existence and uniqueness of the barycenter in the case where $Q$ is finitely supported follows from the seminal work of Agueh and Carlier~\cite{agueh2011barycenter}. We extend this result to the case where $Q$ is not finitely supported.

\begin{pro}[Gaussian barycenter]\label{prop:gaussian_barycenter}
    Fix $0 < \lambda_{\min} \le \lambda_{\max} < \infty$.
    Let $Q \in \mc P_2(\mc P_{2,\rm ac}(\R^D))$ be such that for all $\mu \in \supp Q$, $\mu = \gamma_{m(\mu), \Sigma(\mu)}$ is a Gaussian with $\lambda_{\min} I_D \preceq \Sigma(\mu) \preceq \lambda_{\max} I_D$.
    Let $\gamma_{\bar m,\bar\Sigma}$ be the Gaussian measure with mean $\bar m := \int m(\mu) \, \ud Q(\mu)$ and covariance matrix $\bar\Sigma$ which is a fixed point of the mapping $S \mapsto G(S) := \int {(S^{1/2} \Sigma(\mu) S^{1/2})}^{1/2} \, \ud Q(\mu)$.
    Then, $\gamma_{\bar m,\bar \Sigma}$ is the unique barycenter of $Q$.
\end{pro}
\begin{proof}
    To show that there exists a fixed point for the mapping $G$, apply Brouwer's fixed-point theorem as in~\cite[Theorem 6.1]{agueh2011barycenter}. To see that $\gamma_{\bar m,\bar\Sigma}$ is indeed a barycenter, observe the mapping
    \[ 
    \varphi : (\mu, x) \mapsto \varphi_\mu(x) := \langle x, m(\mu) \rangle + \frac{1}{2} \langle x-\bar m, \bar \Sigma^{-1/2} {[\bar\Sigma^{1/2} \Sigma(\mu) \bar\Sigma^{1/2}]}^{1/2} \bar\Sigma^{-1/2} (x-\bar m)\rangle
    \]
    satisfies the characterization~\eqref{eq:opt_dual_barycenter} (so that $\varphi$ is an optimal dual solution for the barycenter problem w.r.t.\ $\gamma_{\bar m, \bar \Sigma}$) using the explicit form of the transport map~\eqref{eq:gaussian_opt_coupling}, so
    $\gamma_{\bar m,\bar\Sigma}$ is a barycenter of $Q$. Uniqueness follows from the variance inequality (Theorem~\ref{thm:variance_ineq}) once we establish regularity of the optimal transport maps in Lemma~\ref{lem:reg_from_eigvals}.
\end{proof}

\begin{lem}\label{lem:reg_from_eigvals}
    Suppose there exist constants $0 < \lambda_{\min} \le \lambda_{\max} < \infty$ such that all of the eigenvalues of $\Sigma, \Sigma' \in \mbb S_{++}^D$ are bounded between $\lambda_{\min}$ and $\lambda_{\max}$ and define $\kappa=\lambda_{\max}/\lambda_{\min}$.
    Then, the transport map from $\gamma_{0,\Sigma}$ to $\gamma_{0,\Sigma'}$ is $(\kappa^{-1}, \kappa)$-regular.
\end{lem}
\begin{proof}
    The transport map from $\gamma_{0,\Sigma}$ to $\gamma_{0,\Sigma'}$ is the map $x \mapsto \Sigma^{-1/2} {(\Sigma^{1/2} \Sigma' \Sigma^{1/2})}^{1/2} \Sigma^{-1/2}x$. 
 Throughout this proof, we write $\|\cdot\|=\|\cdot\|_{\mathrm{op}}$ for simplicity.   We have the trivial bound
    \begin{align*}
        \norm{\Sigma^{-1/2} {(\Sigma^{1/2} \Sigma' \Sigma^{1/2})}^{1/2} \Sigma^{-1/2}}
        &\le \sqrt{\norm{\Sigma^{-1}} \norm{\Sigma^{1/2} \Sigma' \Sigma^{1/2}} \norm{\Sigma^{-1}}}.
    \end{align*}
Moreover $\norm{\Sigma^{-1}} \le \lambda_{\min}^{-1}$ and $\norm{\Sigma^{1/2} \Sigma' \Sigma^{1/2}} \le \lambda_{\max}^2$, so that the smoothness  is bounded by
    \begin{align*}
        \norm{\Sigma^{-1/2} {(\Sigma^{1/2} \Sigma' \Sigma^{1/2})}^{1/2} \Sigma^{-1/2}}
        &\le \frac{\lambda_{\max}}{\lambda_{\min}}.
    \end{align*}
    We can take advantage of the fact that $\Sigma$, $\Sigma'$ are interchangeable and infer that the strong convexity parameter of the transport map from $\Sigma$ to $\Sigma'$ is the inverse of the smoothness parameter of the transport map from $\Sigma'$ to $\Sigma$. In other words,
    \begin{align*}
      \min_{1\le j \le D}  \lambda_j\bigl(\Sigma^{-1/2} {(\Sigma^{1/2} \Sigma' \Sigma^{1/2})}^{1/2} \Sigma^{-1/2} \bigr)
        &\ge \frac{\lambda_{\min}}{\lambda_{\max}}.
    \end{align*}
    This concludes the proof.
\end{proof}
Theorem~\ref{thm:variance_ineq} readily yields the following variance inequality.
\begin{thm}\label{thm:varineq_bures}
Fix $\sep>0$ and assume that $Q$ is $\sep$-regular. Then $Q$ has a unique barycenter $\bar b$ and it satisfes a variance inequality with constant $C_{\mathsf{var}}=\sep$, that is, for any $b \in \mc P_{2,\rm ac}(\R^D)$, 
$$
G(b)-G(\bar b)\ge \frac{\sep}{2}W_2^2(b,\bar b)\,.
$$
\end{thm}

\subsection{Generalized geodesic convexity of $\ln \|\cdot\|_{L^\infty}$}

\begin{lem}\label{lem:geod_cvx_log_density}
    Identify measures $\rho \in \mc P_{2,\rm ac}(\R^D)$ with their densities, and  let the $\norm \cdot_{L^\infty}$ norm denote the $L^\infty$-norm (essential supremum) w.r.t.\ the Lebesgue measure on $\R^D$.
    Then, for any $b,\mu_0,\mu_1 \in \mc P_{2,\rm ac}(\R^D)$, any $s \in [0,1]$, and almost every $x\in\R^D$, it holds that
    \begin{align*}
        \ln \mu_s^b\bigl(\nabla \phi_{b\to\mu_s^b}(x)\bigr) \le (1-s)\ln \mu_0\bigl(\nabla \phi_{b\to\mu_0}(x)\bigr)+ s\ln \mu_1\bigl(\nabla \phi_{b\to\mu_1}(x)\bigr).
    \end{align*}
    In particular, taking the essential supremum over $x$ on both sides, we deduce that the functional $\mc P_{2,\rm ac}(\R^D) \to (-\infty, \infty]$ given by $\rho \mapsto \ln{\norm{\rho}_{L^\infty}}$ is convex along generalized geodesics.
\end{lem}
\begin{proof}
    Let $\rho := {[(1-s) T_{b\to \mu} + sT_{b\to\nu}]}_\# b$ be a point on the generalized geodesic with base $b$ connecting $\mu$ to $\nu$.
    Let $\phi_{b\to\mu}$, $\phi_{b\to\nu}$ be the convex potentials whose gradients are $T_{b\to\mu}$ and $T_{b\to\nu}$ respectively.
    Then, for almost all $x\in\R^D$, the Monge-Amp\`ere equation applied to the pairs $(b,\mu)$, $(b,\nu)$, and $(b,\rho)$ respectively, yields
    \renewcommand{\arraystretch}{1.5}
$$
        b(x)=\left\{\begin{array}{l}
        \mu\bigl(\nabla \phi_{b\to\mu}(x)\bigr)\det D^2_{\rm A} \phi_{b\to\mu}(x)\\
         \nu\bigl(\nabla \phi_{b\to\nu}(x)\bigr)\det D^2_{\rm A} \phi_{b\to\nu}(x) \label{eq:density_bd_eq1} \\
         \rho\bigl((1-s)\nabla \phi_{b\to\mu}(x) + s \nabla \phi_{b\to\nu}(x)\bigr)\det\bigl((1-s)D^2_{\rm A} \phi_{b\to\mu}(x) + s D^2_{\rm A} \phi_{b\to\nu}(x)\bigr).
    \end{array}\right.
    $$
    Here, $D_{\rm A}^2 \phi$ denotes the Hessian of $\phi$ in the Alexandrov sense; see~\cite[Theorem 4.8]{villani2003topics}.
    
 Fix $x$ such that $b(x)>0$. 
    On the one hand, applying log-concavity of the determinant, it follows from the third Monge-Amp\`ere equation that
    \begin{align*}
        \ln b(x)
        &= \ln \rho\bigl((1-s)\nabla \phi_{b\to\mu}(x) + s \nabla \phi_{b\to\nu}(x)\bigr) + \ln \det\bigl((1-s)D^2_{\rm A} \phi_{b\to\mu}(x) + sD^2_{\rm A} \phi_{b\to\nu}(x)\bigr) \\
        &\ge \ln \rho\bigl((1-s)\nabla \phi_{b\to\mu}(x) + s \nabla \phi_{b\to\nu}(x)\bigr) + (1-s) \ln \det D^2_{\rm A} \phi_{b\to\mu}(x) + s \ln \det D^2_{\rm A} \phi_{b\to\nu}(x).
    \end{align*}
  On the other hand, it follows from the first two Monge-Amp\`ere equations that 
\begin{align*}
    \ln b(x) &=(1-s)\ln \mu\bigl(\nabla \phi_{b\to\mu}(x)\bigr)+ s\ln \nu\bigl(\nabla \phi_{b\to\nu}(x)\bigr) \\
&\qquad\qquad {} +(1-s) \ln \det D^2_{\rm A} \phi_{b\to\mu}(x) + s \ln \det D^2_{\rm A} \phi_{b\to\nu}(x).
\end{align*}
The above two displays yield
$$
\ln \rho\bigl((1-s)\nabla \phi_{b\to\mu}(x) + s \nabla \phi_{b\to\nu}(x)\bigr) \le (1-s)\ln \mu\bigl(\nabla \phi_{b\to\mu}(x)\bigr)+ s\ln \nu\bigl(\nabla \phi_{b\to\nu}(x)\bigr)
$$
It yields the result.
\end{proof}

\subsection{A PL inequality on the Bures-Wasserstein manifold}
\begin{thm}\label{thm:pl_bures}

Fix $\sep \in (0,1]$, and let $Q$ be a $\sep$-regular distribution.
    Then, the barycenter functional $G$ satisfies the PL inequality with constant $C_{\msf{PL}} = \sep^2/4$ uniformly at all $b \in \cS_\sep$:
$$
        G(b) - G(\bar b)\le \frac{2}{\sep^2} \|\nabla G(b)\|_b^2.
$$
\end{thm}

\begin{proof}
    For any $\gamma_{0,\Sigma} \in \cS_\sep$, the eigenvalues of $\Sigma$ are in  $[\sep,1]$.
    Let ${(\tilde b_s)}_{s\in [0,1]}$ be the constant-speed geodesic between $\tilde b_0 := b := \gamma_{0,\Sigma}$ and $\tilde b_1 := \bar b := \gamma_{0,\bar \Sigma}$.
Combining Lemma~\ref{lem:main} (with an additional use of the Cauchy-Schwarz inequality) and Theorem~\ref{thm:varineq_bures}, we get
    \begin{equation}
        \label{eq:int_proof_PL}
        G(b) - G(\bar b)
        \le \frac{2}{\sep} \int_0^1 \int \norm{\nabla G(b)}_2^2 \, \ud \tilde b_s \, \ud s.
    \end{equation}

Define a random variable $X_s\sim \tilde b_s$ and observe that 
$$
\int \norm{\nabla G(b)}_2^2 \, \ud \tilde b_s  =\E\norm{(\tilde M-I_D)X_s}_2^2\,,
\quad 
\text{where}\ 
\tilde M=\int \Sigma^{-1/2} {(\Sigma^{1/2} S \Sigma^{1/2})}^{1/2} \Sigma^{-1/2} \, \ud Q(\gamma_{0,S}).$$
Moreover, recall that $X_s=sX_1+(1-s)X_0$ where $X_0\sim \tilde b_0$ and $X_1\sim \tilde b_1$ are optimally coupled.
Therefore, by Jensen's inequality, we have for all $s \in [0,1]$,
\begin{align*}
\E\norm{(\tilde M -I_D)X_s}_2^2&\le s\E\norm{(\tilde M-I_D)X_1}_2^2+(1-s)\E\norm{(\tilde M -I_D)X_0}_2^2\le \frac{1}{\sep}\E\norm{(\tilde M -I_D)X_0}_2^2\,,
\end{align*}
where in the second inequality, we used the fact that
\begin{align*}
\E\norm{(\tilde M-I_D)X_1}_2^2=\Tr\bigl(\bar \Sigma (\tilde M-I_D)^2\bigr)
\le \norm{\bar \Sigma \Sigma^{-1}}_{\mathrm{op}} \Tr\bigl(\Sigma {(\tilde M-I_D)}^2\bigr)
\le \frac{1}{\sep}\E\norm{(\tilde M-I_D)X_0}_2^2 \, .
\end{align*}

Together with~\eqref{eq:int_proof_PL}, it yields
$$
G(b) - G(\bar b)\le \frac{2}{\sep^2} \E\norm{(\tilde M -I_D)X_0}_2^2=\frac{2}{\sep^2}\norm{\nabla G(b)}_b^2\,.
$$
\end{proof}

\bibliographystyle{aomalpha}
\bibliography{extracted}

\providecommand{\etalchar}[1]{$^{#1}$}
\providecommand{\bysame}{\leavevmode\hbox to3em{\hrulefill}\thinspace}
\providecommand{\noopsort}[1]{}
\providecommand{\mr}[1]{\href{http://www.ams.org/mathscinet-getitem?mr=#1}{MR~#1}}
\providecommand{\zbl}[1]{\href{http://www.zentralblatt-math.org/zmath/en/search/?q=an:#1}{Zbl~#1}}
\providecommand{\jfm}[1]{\href{http://www.emis.de/cgi-bin/JFM-item?#1}{JFM~#1}}
\providecommand{\arxiv}[1]{\href{http://www.arxiv.org/abs/#1}{arXiv~#1}}
\providecommand{\doi}[1]{\url{http://dx.doi.org/#1}}
\providecommand{\MR}{\relax\ifhmode\unskip\space\fi MR }
\providecommand{\MRhref}[2]{%
  \href{http://www.ams.org/mathscinet-getitem?mr=#1}{#2}
}
\providecommand{\href}[2]{#2}
\begin{thebibliography}{AEdBCAM16}

\bibitem[AC11]{agueh2011barycenter}
\bgroup\scshape{}M.~Agueh\egroup{} and \bgroup\scshape{}G.~Carlier\egroup{},
  Barycenters in the {W}asserstein space,  \emph{SIAM J. Math. Anal.}
  \textbf{43} (2011), 904--924.

\bibitem[ACLGP20]{ahidarcoutrix2018convergence}
\bgroup\scshape{}A.~Ahidar-Coutrix\egroup{},
  \bgroup\scshape{}T.~Le~Gouic\egroup{}, and
  \bgroup\scshape{}Q.~Paris\egroup{}, Convergence rates for empirical
  barycenters in metric spaces: curvature, convexity and extendable geodesics,
  \emph{Probab. Theory Related Fields} \textbf{177} (2020), 323--368.

\bibitem[AGS08]{ambrosio2008gradient}
\bgroup\scshape{}L.~Ambrosio\egroup{}, \bgroup\scshape{}N.~Gigli\egroup{}, and
  \bgroup\scshape{}G.~Savar\'{e}\egroup{}, \emph{Gradient flows in metric
  spaces and in the space of probability measures}, second ed., \emph{Lectures
  in Mathematics ETH Z\"{u}rich}, Birkh\"{a}user Verlag, Basel, 2008.

\bibitem[AMR05]{AudMazRuh05}
\bgroup\scshape{}C.~Auderset\egroup{}, \bgroup\scshape{}C.~Mazza\egroup{}, and
  \bgroup\scshape{}E.~A. Ruh\egroup{}, Angular {G}aussian and {C}auchy
  estimation,  \emph{Journal of Multivariate Analysis} \textbf{93} (2005),
  180--197.

\bibitem[Bac14]{Bac14}
\bgroup\scshape{}M.~Bacak\egroup{}, \emph{{Convex analysis and optimization in
  Hadamard spaces}}, \emph{De Gruyter Series in Nonlinear Analysis and
  Applications}, De Gruyter, Berlin, 2014.

\bibitem[BFRT18]{backhoffveraguas2018barycenters}
\bgroup\scshape{}J.~{Backhoff-Veraguas}\egroup{},
  \bgroup\scshape{}J.~{Fontbona}\egroup{}, \bgroup\scshape{}G.~{Rios}\egroup{},
  and \bgroup\scshape{}F.~{Tobar}\egroup{}, {Bayesian learning with
  {W}asserstein barycenters},  \emph{arXiv e-prints} (2018).

\bibitem[BJL19]{bhatia2019bures}
\bgroup\scshape{}R.~Bhatia\egroup{}, \bgroup\scshape{}T.~Jain\egroup{}, and
  \bgroup\scshape{}Y.~Lim\egroup{}, On the {B}ures--{W}asserstein distance
  between positive definite matrices,  \emph{Expo. Math.} \textbf{37} (2019),
  165--191.

\bibitem[BGKL18]{bigot2018barycenter}
\bgroup\scshape{}J.~Bigot\egroup{}, \bgroup\scshape{}R.~Gouet\egroup{},
  \bgroup\scshape{}T.~Klein\egroup{}, and
  \bgroup\scshape{}A.~L\'{o}pez\egroup{}, Upper and lower risk bounds for
  estimating the {W}asserstein barycenter of random measures on the real line,
  \emph{Electron. J. Stat.} \textbf{12} (2018), 2253--2289.

\bibitem[{Bon}13]{Bon13}
\bgroup\scshape{}S.~{Bonnabel}\egroup{}, Stochastic gradient descent on
  {R}iemannian manifolds,  \emph{IEEE Transactions on Automatic Control}
  \textbf{58} (2013), 2217--2229.

\bibitem[BPC16]{bonneel2016wasserstein}
\bgroup\scshape{}N.~Bonneel\egroup{}, \bgroup\scshape{}G.~Peyr{\'e}\egroup{},
  and \bgroup\scshape{}M.~Cuturi\egroup{}, Wasserstein barycentric coordinates:
  Histogram regression using optimal transport,  \emph{ACM Transactions on
  Graphics} \textbf{35} (2016).

\bibitem[BL06]{borwein2006convex}
\bgroup\scshape{}J.~M. Borwein\egroup{} and \bgroup\scshape{}A.~S.
  Lewis\egroup{}, \emph{Convex analysis and nonlinear optimization}, second
  ed., \emph{CMS Books in Mathematics/Ouvrages de Math\'{e}matiques de la SMC}
  \textbf{3}, Springer, New York, 2006, Theory and examples.

\bibitem[Bub15]{bubeck2015convex}
\bgroup\scshape{}S.~Bubeck\egroup{}, Convex optimization: Algorithms and
  complexity,  \emph{Foundations and Trends® in Machine Learning} \textbf{8}
  (2015), 231--357.

\bibitem[BBI01]{buragoivanov2001metricgeometry}
\bgroup\scshape{}D.~Burago\egroup{}, \bgroup\scshape{}Y.~Burago\egroup{}, and
  \bgroup\scshape{}S.~Ivanov\egroup{}, \emph{A course in metric geometry},
  \emph{Graduate Studies in Mathematics} \textbf{33}, American Mathematical
  Society, Providence, RI, 2001.

\bibitem[Bur69]{bures1969extension}
\bgroup\scshape{}D.~Bures\egroup{}, An extension of {K}akutani's theorem on
  infinite product measures to the tensor product of semifinite w*-algebras,
  \emph{Transactions of the American Mathematical Society} \textbf{135} (1969),
  199--212.

\bibitem[CE10]{carlier2010matching}
\bgroup\scshape{}G.~Carlier\egroup{} and \bgroup\scshape{}I.~Ekeland\egroup{},
  Matching for teams,  \emph{Economic Theory} \textbf{42} (2010), 397--418.

\bibitem[dC92]{docarmo1992riemannian}
\bgroup\scshape{}M.~P.~a. do~Carmo\egroup{}, \emph{Riemannian geometry},
  \emph{Mathematics: Theory \& Applications}, Birkh\"{a}user Boston, Inc.,
  Boston, MA, 1992, Translated from the second Portuguese edition by Francis
  Flaherty.

\bibitem[CCS18]{claici2018stochastic}
\bgroup\scshape{}S.~Claici\egroup{}, \bgroup\scshape{}E.~Chien\egroup{}, and
  \bgroup\scshape{}J.~Solomon\egroup{}, Stochastic {W}asserstein barycenters,
  in \emph{Proceedings of the 35th International Conference on Machine
  Learning} (\bgroup\scshape{}J.~Dy\egroup{} and
  \bgroup\scshape{}A.~Krause\egroup{}, eds.), \emph{Proceedings of Machine
  Learning Research} \textbf{80}, Stockholmsmässan, Stockholm Sweden, 2018,
  pp.~999--1008.

\bibitem[CD14]{cuturi2014barycenters}
\bgroup\scshape{}M.~Cuturi\egroup{} and \bgroup\scshape{}A.~Doucet\egroup{},
  Fast computation of {W}asserstein barycenters,  in \emph{Proceedings of the
  31st International Conference on Machine Learning} (\bgroup\scshape{}E.~P.
  Xing\egroup{} and \bgroup\scshape{}T.~Jebara\egroup{}, eds.), \textbf{32}
  Proceedings of Machine Learning Research no.~2, Bejing, China, 2014,
  pp.~685--693.

\bibitem[{Dvi}20]{Dvi20}
\bgroup\scshape{}D.~{Dvinskikh}\egroup{}, {Stochastic approximation versus
  sample average approximation for population {W}asserstein barycenter
  calculation},  \emph{arXiv e-prints} (2020).

\bibitem[AEdBCAM16]{esteban2016barycenters}
\bgroup\scshape{}P.~C. \'{A}lvarez Esteban\egroup{}, \bgroup\scshape{}E.~del
  Barrio\egroup{}, \bgroup\scshape{}J.~A. Cuesta-Albertos\egroup{}, and
  \bgroup\scshape{}C.~Matr\'{a}n\egroup{}, A fixed-point approach to
  barycenters in {W}asserstein space,  \emph{J. Math. Anal. Appl.} \textbf{441}
  (2016), 744--762.

\bibitem[GPC15]{gramfort2015fast}
\bgroup\scshape{}A.~Gramfort\egroup{}, \bgroup\scshape{}G.~Peyr{\'e}\egroup{},
  and \bgroup\scshape{}M.~Cuturi\egroup{}, Fast optimal transport averaging of
  neuroimaging data,  in \emph{International Conference on Information
  Processing in Medical Imaging}, Springer, 2015, pp.~261--272.

\bibitem[GDTG19]{guminov2019accelerated}
\bgroup\scshape{}S.~{Guminov}\egroup{},
  \bgroup\scshape{}P.~{Dvurechensky}\egroup{},
  \bgroup\scshape{}N.~{Tupitsa}\egroup{}, and
  \bgroup\scshape{}A.~{Gasnikov}\egroup{}, {Accelerated alternating
  minimization, accelerated {S}inkhorn's algorithm and accelerated iterative
  {B}regman projections},  \emph{arXiv e-prints} (2019).

\bibitem[HGA15]{huang2015broyden}
\bgroup\scshape{}W.~Huang\egroup{}, \bgroup\scshape{}K.~A. Gallivan\egroup{},
  and \bgroup\scshape{}P.-A. Absil\egroup{}, A {B}royden class of
  quasi-{N}ewton methods for {R}iemannian optimization,  \emph{SIAM Journal on
  Optimization} \textbf{25} (2015), 1660--1685.

\bibitem[KNS16]{karimi2016linear}
\bgroup\scshape{}H.~Karimi\egroup{}, \bgroup\scshape{}J.~Nutini\egroup{}, and
  \bgroup\scshape{}M.~Schmidt\egroup{}, Linear convergence of gradient and
  proximal-gradient methods under the {P}olyak-{L}ojasiewicz condition,  in
  \emph{Joint European Conference on Machine Learning and Knowledge Discovery
  in Databases}, Springer, 2016, pp.~795--811.

\bibitem[KS94]{knott1994generalization}
\bgroup\scshape{}M.~Knott\egroup{} and \bgroup\scshape{}C.~S. Smith\egroup{},
  On a generalization of cyclic monotonicity and distances among random
  vectors,  \emph{Linear Algebra and its Applications} \textbf{199} (1994),
  363--371.

\bibitem[KSS19]{kroshnin2019barycenters}
\bgroup\scshape{}A.~{Kroshnin}\egroup{},
  \bgroup\scshape{}V.~{Spokoiny}\egroup{}, and
  \bgroup\scshape{}A.~{Suvorikova}\egroup{}, {Statistical inference for
  {B}ures-{W}asserstein barycenters},  \emph{arXiv e-prints} (2019).

\bibitem[KTD{\etalchar{+}}19]{pmlr-v97-kroshnin19a}
\bgroup\scshape{}A.~Kroshnin\egroup{}, \bgroup\scshape{}N.~Tupitsa\egroup{},
  \bgroup\scshape{}D.~Dvinskikh\egroup{},
  \bgroup\scshape{}P.~Dvurechensky\egroup{},
  \bgroup\scshape{}A.~Gasnikov\egroup{}, and
  \bgroup\scshape{}C.~Uribe\egroup{}, On the complexity of approximating
  {W}asserstein barycenters,  in \emph{Proceedings of the 36th International
  Conference on Machine Learning} (\bgroup\scshape{}K.~Chaudhuri\egroup{} and
  \bgroup\scshape{}R.~Salakhutdinov\egroup{}, eds.), \emph{Proceedings of
  Machine Learning Research} \textbf{97}, PMLR, Long Beach, California, USA,
  09--15 Jun 2019, pp.~3530--3540.

\bibitem[LG20]{legouic2020barycenter}
\bgroup\scshape{}T.~Le~Gouic\egroup{}, Dual and multimarginal problems for the
  {W}asserstein barycenter,  (2020), Unpublished.

\bibitem[LGL17]{legouic2015consistency}
\bgroup\scshape{}T.~Le~Gouic\egroup{} and \bgroup\scshape{}J.-M.
  Loubes\egroup{}, Existence and consistency of {W}asserstein barycenters,
  \emph{Probab. Theory Related Fields} \textbf{168} (2017), 901--917.

\bibitem[LL20]{legouic2020fairness}
\bgroup\scshape{}T.~{Le Gouic}\egroup{} and \bgroup\scshape{}J.-M.
  {Loubes}\egroup{}, {The price for fairness in a regression framework},
  \emph{arXiv e-prints} (2020).

\bibitem[LPRS19]{legouic2019fast}
\bgroup\scshape{}T.~{Le Gouic}\egroup{}, \bgroup\scshape{}Q.~{Paris}\egroup{},
  \bgroup\scshape{}P.~{Rigollet}\egroup{}, and \bgroup\scshape{}A.~J.
  {Stromme}\egroup{}, {Fast convergence of empirical barycenters in
  {A}lexandrov spaces and the {W}asserstein space},  \emph{arXiv e-prints}
  (2019).

\bibitem[LHC{\etalchar{+}}20]{lin2020fixedsupport}
\bgroup\scshape{}T.~{Lin}\egroup{}, \bgroup\scshape{}N.~{Ho}\egroup{},
  \bgroup\scshape{}X.~{Chen}\egroup{}, \bgroup\scshape{}M.~{Cuturi}\egroup{},
  and \bgroup\scshape{}M.~I. {Jordan}\egroup{}, {Fixed-support {W}asserstein
  barycenters: Computational hardness and fast algorithm},  \emph{arXiv
  e-prints} (2020).

\bibitem[LV09]{lott2009ricci}
\bgroup\scshape{}J.~Lott\egroup{} and \bgroup\scshape{}C.~Villani\egroup{},
  Ricci curvature for metric-measure spaces via optimal transport,  \emph{Ann.
  of Math. (2)} \textbf{169} (2009), 903--991.

\bibitem[MMP18]{malago2018wasserstein}
\bgroup\scshape{}L.~Malag{\`o}\egroup{},
  \bgroup\scshape{}L.~Montrucchio\egroup{}, and
  \bgroup\scshape{}G.~Pistone\egroup{}, Wasserstein {R}iemannian geometry of
  {G}aussian densities,  \emph{Information Geometry} \textbf{1} (2018),
  137--179.

\bibitem[MC17]{agueh2017centrale}
\bgroup\scshape{}A.~Martial\egroup{} and \bgroup\scshape{}G.~Carlier\egroup{},
  Vers un th\'{e}or\`eme de la limite centrale dans l'espace de {W}asserstein?,
   \emph{C. R. Math. Acad. Sci. Paris} \textbf{355} (2017), 812--818.

\bibitem[Mod17]{modin2017matrix}
\bgroup\scshape{}K.~Modin\egroup{}, Geometry of matrix decompositions seen
  through optimal transport and information geometry,  \emph{J. Geom. Mech.}
  \textbf{9} (2017), 335--390.

\bibitem[Ott01]{otto2001geometry}
\bgroup\scshape{}F.~Otto\egroup{}, The geometry of dissipative evolution
  equations: the porous medium equation,  \emph{Comm. Partial Differential
  Equations} \textbf{26} (2001), 101--174.

\bibitem[PZ16]{panaretos2016point}
\bgroup\scshape{}V.~M. Panaretos\egroup{} and
  \bgroup\scshape{}Y.~Zemel\egroup{}, Amplitude and phase variation of point
  processes,  \emph{Ann. Statist.} \textbf{44} (2016), 771--812.

\bibitem[RP15]{rabin_papadakis_2015}
\bgroup\scshape{}J.~Rabin\egroup{} and \bgroup\scshape{}N.~Papadakis\egroup{},
  Convex color image segmentation with optimal transport distances,  in
  \emph{International Conference on Scale Space and Variational Methods in
  Computer Vision}, Springer, 2015, pp.~256--269.

\bibitem[RPDB11]{rabin2011wasserstein}
\bgroup\scshape{}J.~Rabin\egroup{}, \bgroup\scshape{}G.~Peyr{\'e}\egroup{},
  \bgroup\scshape{}J.~Delon\egroup{}, and \bgroup\scshape{}M.~Bernot\egroup{},
  Wasserstein barycenter and its application to texture mixing,  in
  \emph{International Conference on Scale Space and Variational Methods in
  Computer Vision}, Springer, 2011, pp.~435--446.

\bibitem[Roc97]{rockafellar1997convexanalysis}
\bgroup\scshape{}R.~T. Rockafellar\egroup{}, \emph{Convex analysis},
  \emph{Princeton Landmarks in Mathematics}, Princeton University Press,
  Princeton, NJ, 1997, Reprint of the 1970 original, Princeton Paperbacks.

\bibitem[San15]{San15}
\bgroup\scshape{}F.~Santambrogio\egroup{}, \emph{Optimal transport for applied
  mathematicians}, \emph{Progress in Nonlinear Differential Equations and their
  Applications} \textbf{87}, Birkh\"auser/Springer, Cham, 2015.

\bibitem[SDGP{\etalchar{+}}15]{solomon2015convolutional}
\bgroup\scshape{}J.~Solomon\egroup{}, \bgroup\scshape{}F.~De~Goes\egroup{},
  \bgroup\scshape{}G.~Peyr{\'e}\egroup{}, \bgroup\scshape{}M.~Cuturi\egroup{},
  \bgroup\scshape{}A.~Butscher\egroup{}, \bgroup\scshape{}A.~Nguyen\egroup{},
  \bgroup\scshape{}T.~Du\egroup{}, and \bgroup\scshape{}L.~Guibas\egroup{},
  Convolutional {W}asserstein distances: Efficient optimal transportation on
  geometric domains,  \emph{ACM Transactions on Graphics (TOG)} \textbf{34}
  (2015), 1--11.

\bibitem[SLD18]{srivastava2018scalable}
\bgroup\scshape{}S.~Srivastava\egroup{}, \bgroup\scshape{}C.~Li\egroup{}, and
  \bgroup\scshape{}D.~B. Dunson\egroup{}, Scalable {B}ayes via barycenter in
  {W}asserstein space,  \emph{The Journal of Machine Learning Research}
  \textbf{19} (2018), 312--346.

\bibitem[Stu03]{sturmnpc}
\bgroup\scshape{}K.-T. Sturm\egroup{}, Probability measures on metric spaces of
  nonpositive curvature,  in \emph{Heat kernels and analysis on manifolds,
  graphs, and metric spaces ({P}aris, 2002)}, \emph{Contemp. Math.}
  \textbf{338}, Amer. Math. Soc., Providence, RI, 2003, pp.~357--390.

\bibitem[TFBJ18]{tripuraneni2018averaging}
\bgroup\scshape{}N.~Tripuraneni\egroup{},
  \bgroup\scshape{}N.~Flammarion\egroup{}, \bgroup\scshape{}F.~Bach\egroup{},
  and \bgroup\scshape{}M.~I. Jordan\egroup{}, Averaging stochastic gradient
  descent on {R}iemannian manifolds,  in \emph{Proceedings of the 31st
  Conference On Learning Theory} (\bgroup\scshape{}S.~Bubeck\egroup{},
  \bgroup\scshape{}V.~Perchet\egroup{}, and
  \bgroup\scshape{}P.~Rigollet\egroup{}, eds.), \emph{Proceedings of Machine
  Learning Research} \textbf{75}, 2018, pp.~650--687.

\bibitem[Vil03]{villani2003topics}
\bgroup\scshape{}C.~Villani\egroup{}, \emph{Topics in optimal transportation},
  \emph{Graduate Studies in Mathematics} \textbf{58}, American Mathematical
  Society, Providence, RI, 2003.

\bibitem[Vil09]{villani2009optimal}
\bgroup\scshape{}C.~Villani\egroup{}, \emph{Optimal transport: old and new},
  \emph{Grundlehren der Mathematischen Wissenschaften [Fundamental Principles
  of Mathematical Sciences]} \textbf{338}, Springer-Verlag, Berlin, 2009.

\bibitem[WS17]{webersra2017rfw}
\bgroup\scshape{}M.~{Weber}\egroup{} and \bgroup\scshape{}S.~{Sra}\egroup{},
  {Riemannian optimization via {F}rank-{W}olfe methods},  \emph{arXiv e-prints}
  (2017).

\bibitem[WS19]{weber2019nonconvex}
\bgroup\scshape{}M.~{Weber}\egroup{} and \bgroup\scshape{}S.~{Sra}\egroup{},
  {Projection-free nonconvex stochastic optimization on {R}iemannian
  manifolds},  \emph{arXiv e-prints} (2019).

\bibitem[{Wie}12]{Wie12}
\bgroup\scshape{}A.~{Wiesel}\egroup{}, Geodesic convexity and covariance
  estimation,  \emph{IEEE Transactions on Signal Processing} \textbf{60}
  (2012), 6182--6189.

\bibitem[ZP19]{zemel2019procrustes}
\bgroup\scshape{}Y.~Zemel\egroup{} and \bgroup\scshape{}V.~M.
  Panaretos\egroup{}, Fr\'{e}chet means and {P}rocrustes analysis in
  {W}asserstein space,  \emph{Bernoulli} \textbf{25} (2019), 932--976.

\bibitem[ZS16]{zhangSra16a}
\bgroup\scshape{}H.~Zhang\egroup{} and \bgroup\scshape{}S.~Sra\egroup{},
  First-order methods for geodesically convex optimization,  in \emph{29th
  Annual Conference on Learning Theory} (\bgroup\scshape{}V.~Feldman\egroup{},
  \bgroup\scshape{}A.~Rakhlin\egroup{}, and
  \bgroup\scshape{}O.~Shamir\egroup{}, eds.), \emph{Proceedings of Machine
  Learning Research} \textbf{49}, Columbia University, New York, New York, USA,
  2016, pp.~1617--1638.

\end{thebibliography}

\end{document}